\documentclass{siamltex}		
\usepackage[utf8]{inputenc}
\usepackage{graphicx,multirow}  
\usepackage{verbatim,float} 
\usepackage{todonotes}
\usepackage{amssymb,amsmath}
\usepackage[T1]{fontenc}
\usepackage{graphicx,amsmath,amsfonts,amssymb,subfigure}
\usepackage{color}
\usepackage{hyperref}
\usepackage[mathscr]{euscript}
\usepackage{algorithm,algorithmic}

\newcommand{\C}{\mathbb{C}}

\def\R{\mathcal{R}}

\newcommand{\eq}[1]{\begin{equation}\label{#1}}
\newcommand{\en}{\end{equation}}
\usepackage{multirow}

\graphicspath{{./FIGS/}}

\usepackage{verbatim}

\newtheorem{remark}{Remark}[section]

\newtheorem{exe}{Exemple}
\title{Tensor Krylov subspace methods via  the T-product for color  image processing}
\author{M. El Guide \thanks{Centre for Behavioral Economics and Decision Making(CBED), FGSES, Mohammed VI Polytechnic University, Green City, Morocco} \and A. El Ichi\footnotemark [3]\thanks{Department of Mathematics University Mohammed V Rabat, Morocco}   \and K. Jbilou\footnotemark[1] \thanks{LMPA, 50 rue F. Buisson, ULCO Calais, France; Mohammed VI Polytechnic University, Green City, Morocco; jbilou@univ-littoral.fr } \and R. Sadaka \footnotemark[2] }

\begin{document} 

\maketitle 
\begin{abstract}
  The present paper is concerned with  developing tensor iterative Krylov subspace methods to solve large
  multi-linear tensor equations. We use the well known T-product for two tensors to define  tensor global Arnoldi and
  tensor global  Gloub-Kahan bidiagonalization algorithms. Furthermore,  we illustrate how tensor–based global approaches can be exploited
  to solve ill-posed problems arising from
  recovering blurry multichannel
  (color) images  and videos, using the so-called Tikhonov regularization technique, to provide computable approximate
  regularized solutions. We also review a generalized cross-validation and discrepancy principle type  of criterion for the selection of
  the regularization parameter in the Tikhonov regularization. Applications
  to RGB image and video processing are given to demonstrate the efficiency of the algorithms.
\end{abstract}



\section{Background and introduction}
The aim of this paper is to solve the following tensor problem

\begin{equation}\label{eq1}
{\mathcal{M}} (\mathscr{X}) = \mathscr{C},
\end{equation}
where  ${\mathcal M}$ is a linear operator  that could be described as 
\begin{equation}\label{eq2}
{\mathcal M} (\mathscr{X}) = \mathscr{A} \ast\mathscr{X},
\end{equation}
or as
\begin{equation}\label{eql}
{\mathcal M} (\mathscr{X}) = \mathscr{A} \ast\mathscr{X}\ast\mathscr{B},
\end{equation}
where $\mathscr{A}$, $\mathscr{X}$, $\mathscr{B}$ and $\mathscr{C}$ are  three-way tensors, leaving the specific dimensions to be defined later,  and $\ast$ is the T-product to be also defined later. To mention but a few applications, problems of these types arise
in  engineering \cite{qllz}, signal processing \cite{lb},  data mining \cite{lxnm}, tensor complementarity problems\cite{lzqlxn}, computer vision\cite{vt1, vt2} and graph analysis \cite{kolda2}.  For those applications, and so many more, one  have to take advantage of this multidimensional structure to build rapid and robust iterative methods for solving large-scale problems. We will then, be  interested
in  developing robust and fast iterative  tensor  Krylov subspace methods under tensor-tensor product framework between third-order
tensors, to solve regularized problems originating from color image and  video processing applications. Standard and global Krylov subspace methods are suitable  when dealing with grayscale images, e.g, \cite{belguide, belguide2, reichel1, reichel2}. However, these methods might be time consuming to numerically solve problems related to multi channel images (e.g. color images, hyper-spectral images and videos). \\

For the Einstein product, both the Einstein tensor global   Arnoldi and Einstein
tensor global  Gloub-Kahan bidiagonalization algorithms have been established \cite{Elguide}, which makes so  natural to generalize these methods using the T-product.  In this paper, we will show that the
tensor-tensor product  between third-order
tensors allows the application of the global iterative methods, such as the global Arnoldi and global Golub-Kahan algorithms. The tensor form of the proposed Krylov methods, together with using the fast
Fourier transform (FFT) to compute the T-product between third-order
tensors can be efficiently implemented on many modern computers and allows to significantly reduce the overall computational complexity. It is also worth mentioning that our approaches can be naturally generalized to higher-order tensors in a recursive manner.

Our paper is organized as follows. We shall first   present in  Section 2 
some symbols and notations used throughout paper. We also recall the concept T-product  between two tensors.  In Section 3, we define tensor global  Arnoldi and tensor global Golub-Kahan algorithms that allow the use of the T-product. Section 4 reviews the adaptation of Tikhonov regularization for  tensor equation (\ref{eq1}) and then proposing a restarting strategy of  the so-called  tensor global GMRES  and tensor global Golub-Kahan approach in connection with Gauss-type quadrature rules to inexpensively  compute  solution of the regularization of (\ref{eq1}).  In Section 5, we give  a tensor formulation in the form of (\ref{eq1}) that  describes the cross-blurring of color image and  then we present a few numerical examples on restoring blurred and noisy color images and videos. Concluding remarks can be found in Section 6.

\section{Definitions and Notations} A tensor is  a multidimensional array of data. The number of indices of a tensor is called modes or ways. 
Notice that a scalar can be regarded as a zero mode tensor, first mode tensors are vectors and matrices are second mode tensor. The order of a tensor is the dimensionality of the array needed to represent it, also known as
ways or modes. 
For a given N-mode (or order-N) tensor $ \mathscr {X}\in \mathbb{R}^{n_{1}\times n_{2}\times n_{3}\ldots \times n_{N}}$, the notation $x_{i_{1},\ldots,i_{N}}$ (with $1\leq i_{j}\leq n_{j}$ and $ j=1,\ldots N $) stands for the element $\left(i_{1},\ldots,i_{N} \right) $ of the tensor $\mathscr {X}$. The norm of a tensor $\mathscr{A}\in \mathbb{R}^{n_1\times n_2\times \cdots \times n_\ell}$ is specified by
\[
\left\| \mathscr{A} \right\|_F^2 = {\sum\limits_{i_1  = 1}^{n_1 } {\sum\limits_{i_2  = 1}^{n_2 } {\cdots\sum\limits_{i_\ell = 1}^{n_\ell} {a_{i_1 i_2 \cdots i_\ell }^2 } } } }^{}.
\] Corresponding to a given tensor $ \mathscr {A}\in \mathbb{R}^{n_{1}\times n_{2}\times n_{3}\ldots \times n_{N}}$, the notation $$ \mathscr {A}_{\underbrace{::\ldots:}_{(N-1)-\text{ times}}k}\; \; {\rm for }  \quad k=1,2,\ldots,n_{N}$$ denotes a tensor in $\mathbb{R}^{n_{1}\times n_{2}\times n_{3}\ldots \times n_{N-1}}$ which is obtained by fixing the last index and is called frontal slice. Fibers are the higher-order analogue of matrix rows and columns. A fiber is
defined by fixing all the indexes  except  one. A matrix column is a mode-1 fiber and a matrix row is a mode-2 fiber. Third-order tensors have column, row and tube fibers. An element $c\in \mathbb{R}^{1\times 1 \times n}$ is called a tubal-scalar of length $n$. More details are found in  \cite{kimler1,kolda1}.  \\

\begin{figure}[!h]
	\begin{center}
		\includegraphics[scale=0.4]{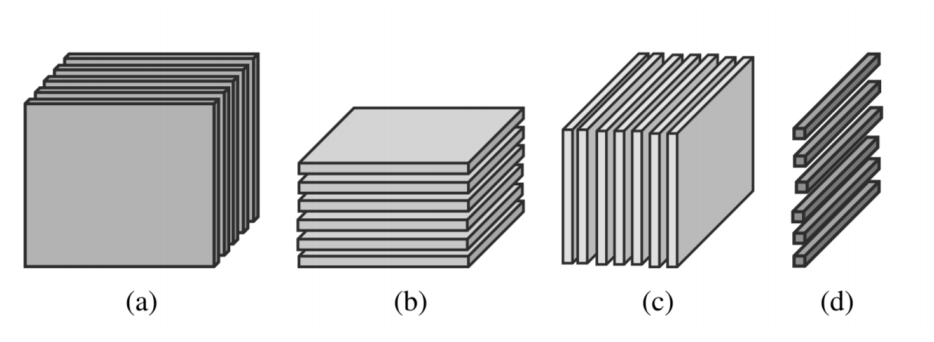}
		\caption{  (a) Frontal, (b) horizontal, and (c) lateral slices of a third   order tensor.   (d) A  mode-3 tube fibers. }
		\label{fig:fibre111}
	\end{center}
\end{figure}	 
\subsection{Discrete Fourier Transformation}
In this subsection we recall some definitions and properties of the discrete Fourier transformation and the T-product. The Discrete Fourier Transformation (DFT) plays a very important role in the definition of the T-product of tensors. The DFT on 
a vector $v \in {\R}^n$ is defined by
\begin{equation}
\label{dft1}
\tilde v= F_n(v) \in {\C}^n,
\end{equation}
where $F_n$ is the matrix defined as 
\begin{equation}
\label{dft2}
F_n(v) = \left (
\begin{array}{ccccc}
1 & 1 & 1 &\ldots&1\\
1 & \omega & \omega^2&\ldots&\omega^{n-1}\\
\vdots & \vdots & \vdots &\ldots&\vdots\\
1 & \omega^{n-1}& \omega^{2(n-1)} &\ldots&\omega^{(n-1)(n-1)}\\
\end{array}
\right) \in {\C}^{n \times n},
\end{equation}
where $\omega=e^{\frac{-2 \pi i}{n}}$ with $i^2=-1$. It is not difficult to show that (see \cite{golub1}) 
\begin{equation}
\label{dft3}
F_n^{*} =\overline F_n,\; {\rm and}\; 	F_n^{*} F_n=F_n F_n^{*} =nI_n.
\end{equation}
Then $F_n^{-1}= \displaystyle \frac{1}{n} \overline F_n$ which show that  $\displaystyle \frac{1}{\sqrt {n}}F_n$ is a unitary matrix.\\ The cost of computing the vector $\tilde v$ directly from \eqref{dft1} is $O(n^2)$. Using the Fast Fourier Transform ({\tt fft}), it will costs $O(nlog(n))$. It is known that 
\begin{equation}\label{dft5}
F_n \, {\rm circ}(v)\, F_n^{-1}  = {\rm Diag}(\tilde v),\\
\end{equation}
which is equivalent to
\begin{equation}\label{dft6}
F_n \, {\rm circ}(v)\, F_n^{*}  = n\,{\rm Diag}(\tilde v),\\
\end{equation}
where
$$
{\rm circ}(v)= \left ( 
\begin{array}{cccc}
v_1 & v_2 & \ldots & v_n\\
v_2 & v_1 & \ldots & v_3\\
\vdots & \vdots  & \ldots & \vdots \\
v_n & v_{n-1} & \ldots & v_1\\
\end{array}
\right ),
$$
and $ {\rm Diag}(\tilde v),$ is the diagonal matrix whose $i$-th diagonal element is $ {\rm Diag}(\tilde v)_i$. 
The decomposition \eqref{dft5} shows that the columns of $F_n$ are the eigenvectors of ${\rm circ}(v))^T$. 

\subsection{Definitions and properties of the T-product}
In this part, we briefly review some concepts and notations, which play a central role for the elaboration of the global iterative methods based on T-product; see \cite{Braman, Hao,kilmer0,kimler1} for more details. 	Let $\mathscr {A} \in \mathbb{R}^{n_{1}\times n_{2}\times n_{3}} $ be a third-order tensor, then the operations {\tt bcirc},   {\tt unfold} and {\tt fold} are defined by
$$ {\tt bcirc}(\mathscr {A})=\left( {\begin{array}{*{20}{c}}
	{{A_1}}&{{A_{{n_3}}}}&{{A_{{n_{3 - 1}}}}}& \ldots &{{A_2}}\\
	{{A_2}}&{{A_1}}&{{A_{{n_3}}}}& \ldots &{{A_3}}\\
	\vdots & \ddots & \ddots & \ddots & \vdots \\
	{{A_{{n_3}}}}&{{A_{{n_{3 - 1}}}}}& \ddots &{{A_2}}&{{A_1}}
	\end{array}} \right)   \in {\R}^{ n_1n_2 \times n_2n_3},$$ 
$${\tt unfold}(\mathscr {A} ) = \begin{pmatrix}
A_{1}   \\
A_{2}   \\
\vdots \\
A_{n_{3}}\end{pmatrix} \in \mathbb{R}^{n_{2}n_{3}\times m_{2}} ,  \qquad {\tt fold}({\tt unfold}(\mathscr {A}) ) =  \mathscr {A}.$$ Let $\widetilde {\mathscr{A}}$ be the tensor obtained by applying the DFT on all the tubes of the tensor $\mathscr {A}$. With the Matlab command ${\tt fft}$, we have
$$\widetilde {\mathscr{A}}= {\tt fft}(\mathscr {A},[ ],3), \; {\rm and }\;\; {\tt ifft} (\widetilde {\mathscr{A}}, [ ],3)= \mathscr {A},$$
where ${\tt ifft}$ denotes the Inverse Fast Fourier Transform.\\
Let ${\bf A}$ be the matrix 
\begin{equation}\label{dft9}
{\bf A}= \left (
\begin{array}{cccc}
{A}^{(1)}& &&\\
& {A}^{(2)}&&\\
&&\ddots&\\
&&&{A}^{(n_3)}\\
\end{array}
\right),
\end{equation}
and the matrices ${A}^{(i)}$'s are the frontal slices of the tensor ${\widetilde {\mathscr{A}}}$.\\
The block circulant matrix ${\tt bcirc}(\mathscr {A})$ can also be block diagonalized by using the DFT and this gives
\begin{equation}\label{dft8}
(F_{n_3} \otimes I_{n_1})\, {\tt bcirc}(\mathscr {A})\, 	(F_{n_3}^{-1} \otimes I_{n_2})={\bf A},
\end{equation}	 	
As noticed in \cite{kimler1,lu}, the diagonal blocks of the matrix ${\bf A}$ satisfy the following property 
\begin{equation}
\label{f1}
\left \{
\begin{array}{ll}
{A}^{(1)} \in {\R}^{n_1 \times n_2}\\
conj({A}^{(i)})= A^{(n_3-i+2)},\\
\end{array}
\right.
\end{equation} 
where 	$conj ({A}^{(i)})$ is the complex conjugate of the matrix ${A}^{(i)}$. Next we recall the definition of the T-product.	

\medskip

\begin{definition}
	The \textbf{T-product} ($\star $) between  two tensors
	$\mathscr {A} \in \mathbb{R}^{n_{1}\times n_{2}\times n_{3}} $ and $\mathscr {B} \in \mathbb{R}^{n_{2}\times m\times n_{3}} $ is an  ${n_{1}\times m\times n_{3}}$ tensor  given by:	
	$$\mathscr {A} \star \mathscr {B}={\rm fold}({\rm bcirc}(\mathscr {A}){\rm unfold}(\mathscr {B}) ).$$
\end{definition}
Notice that from the relation \eqref{dft9}, we can show that the the product $\mathscr {C}= \mathscr {A} \star \mathscr {B}$ is equivalent to $  {\bf C}= {\bf A}\,{\bf B}$. So, the efficient way to compute the T-product is to use Fast Fourier Transform (FFT). 
Using the relation \eqref{f1}, the following algorithm allows us to compute in an efficient way the T-product of the tensors $\mathscr {A}$ and 
$\mathscr {B}$, see \cite{lu}.\\

\begin{algorithm}[!ht]
	\caption{Computing the  T-product via FFT}\label{algo1}
	Inputs: $\mathscr {A} \in \mathbb{R}^{n_{1}\times n_{2}\times n_{3}} $ and $\mathscr {B} \in \mathbb{R}^{n_{2}\times m\times n_{3}} $\\
	Output: $\mathscr {C}= \mathscr {A} \star \mathscr {B}  \in \mathbb{R}^{n_{1}\times m \times n_{3}} $
	\begin{enumerate}
		\item Compute $\mathscr {\widetilde A}={\tt fft}(\mathscr {A},[ ],3)$ and $\mathscr {\widetilde B}={\tt fft}(\mathscr {B},[ ],3)$.
		\item Compute each frontal slices of $\mathscr {\widetilde C}$ by\\
		$$C^{(i)}= \left \{
		\begin{array}{ll}
		A^{(i)} B^{(i)} , \quad \quad\quad  i=1,\ldots,\lfloor \displaystyle \frac{{n_3}+1}{2} \rfloor\\
		conj({C}^{(n_3+i-2)}),\quad \quad i=\lfloor \displaystyle \frac{{n_3}+1}{2} \rfloor+1,\ldots,n_3 .\\
		\end{array}
		\right.$$
		\item Compute $\mathscr {C}={\tt ifft}(\widetilde {C},[],3)$.	 	
	\end{enumerate}
\end{algorithm}

\noindent For the T-product, we have the following definitions

\begin{definition}
	The identity tensor $\mathscr{I}_{n_{1}n_{1}n_{3}} $ is the tensor whose first frontal slice is the identity matrix $I_{n_1,n_1}$ and the other frontal slices are all zeros.
\end{definition}
\medskip		 	
\begin{definition}
	\begin{enumerate}
		\item 	An $n_{1}\times n_{1} \times n_{3}$ tensor $\mathscr{A}$ is invertible, if there exists a tensor $\mathscr{B}$ of order  $n_{1}\times n_{1} \times n_{3}$  such that
		$$\mathscr{A}  \star \mathscr{B}=\mathscr{I}_{ n_{1}  n_{1}  n_{3}} \qquad \text{and}\qquad \mathscr{B}  \star \mathscr{A}=\mathscr{I}_{ n_{1}  n_{1}  n_{3}}.$$
		In that case, we set $\mathscr{B}=\mathscr{A}^{-1}$. 	It is clear that 	$\mathscr{A}$ is invertible if and only if   ${\rm bcirc}(\mathscr{A})$ is invertible.
		\item The transpose of $\mathscr{A}$  is obtained by transposing each of the frontal slices and then reversing the order of transposed frontal slices 2 through $n_3$. 
		\item If $\mathscr {A}$, $\mathscr {B}$ and $\mathscr {A}$ are tensors of appropriate order, then
		$$(\mathscr {A} \star \mathscr {B}) \star \mathscr {C}= \mathscr {A} \star (\mathscr {B} \star \mathscr {C}).$$
		\item Suppose $\mathscr {A}$ and $\mathscr {B}$ are two tensors such $\mathscr {A} \star \mathscr {B}$ and $ \mathscr {B}^T \star \mathscr {A}^T$ are  defined. Then  $$(\mathscr {A} \star \mathscr {B})^T= \mathscr {B}^T \star \mathscr {A}^T.$$
	\end{enumerate}
\end{definition}

\noindent \begin{exe}
	If $\mathscr{A} \in \mathbb{R}^{n_{1}\times n_{2}\times 5}$ and its frontal slices are given by the $n_{1}\times n_{2}$
	matrices $A_{1}, A_{2}, A_{3}, A_{4}, A_{5}$, then
	$$\mathscr{A}^{T} ={\tt fold} \begin{pmatrix}
	A_{1}^{T}   \\
	A_{5}^{T} \\
	A_{4}^{T} \\
	A_{3}^{T}\\
	A_{2}^{T} \\
	\end{pmatrix}.$$
	
\end{exe}

\begin{definition} Let 	$\mathscr{A}$ and $\mathscr{B}$ two tensors in $\mathbb{R}^{n_1 \times n_2 \times n_3}$. Then
	\begin{enumerate}
		\item The scalar inner product is defined by
		$$\langle \mathscr{A}, \mathscr{B} \rangle = \displaystyle \sum_{i_1=1}^{n_1} \sum_{i_2=1}^{n_2}  \sum_{i_3=1}^{n_3} a_{i_1 i_2 i_3}b_{i_1 i_2 i_3}.$$
		\item The norm of $\mathscr{A}$ is defined by
		$$ \Vert \mathscr{A} \Vert_F=\displaystyle \sqrt{\langle  \mathscr{A} ,  \mathscr{A}  \rangle}.$$
	\end{enumerate}
\end{definition}
\medskip
\begin{remark}
	Another interesting way for computing the scalar product and the associated norm is as follows:
	$$\langle \mathscr{A}, \mathscr{B} \rangle = \displaystyle \frac{1}{n_3}  \langle  {\bf A}, {\bf B} \rangle;\; \; \Vert \mathscr{A} \Vert_F= \displaystyle \frac{1}{\sqrt{n_3}} \Vert {{\bf A}} \Vert_F,$$
	where the block diagonal matrix ${\bf A}$  is defined by \eqref{dft9}.
\end{remark}

\begin{definition}
	An $n_{1}\times n_{1} \times n_{3}$ tensor  $\mathscr{Q}$  is orthogonal if
	$$\mathscr{Q}^{T}   \star  \mathscr{Q}=\mathscr{Q} \star \mathscr{Q}^{T}=\mathscr{I}_{ n_{1}  n_{1}  n_{3}}.$$
\end{definition}

\begin{lemma}
	If $\mathscr{Q}$ is an orthogonal tensor, then
	$$ \left\|\mathscr{Q} \star \mathscr{A}\right\|_F= \left\| \mathscr{A}\right\|_F.$$
\end{lemma}

\begin{definition}\cite{kimler1}
	A tensor is called f-diagonal if its frontal slices are orthogonal matrices. It is called upper triangular if all its frontal slices are upper triangular. 
\end{definition}

\begin{definition}\label{bloctens0}\cite{miaoTfunction}{(Block tensor based on T-product)} 
	Suppose $\mathscr{A}   \in {\mathbb R}^{n_{1}\times m_{1} \times n_{3}} $, $\mathscr{B}\in {\mathbb R}^{n_{1}\times m_{2} \times n_{3}}$, $\mathscr{C}   \in {\mathbb R}^{n_{2}\times m_{1} \times n_{3}} $ and   $\mathscr{D}\in {\mathbb R}^{n_{2}\times m_{2} \times n_{3}}$ are four tensors. The block tensor
	$$\left[ {\begin{array}{*{20}{c}}
		{\mathscr{A}}&{\mathscr{B}} \\
		{\mathscr{C}}&{\mathscr{D}} \\
		\end{array}} \right]\in {\mathbb R}^{(n_{1}+n_2)\times (m_1+m_{2}) \times n_{3}} $$
	is defined by compositing the frontal slices of the four tensors.
\end{definition}

\medskip
\noindent Now we introduce the T-diamond tensor product. \\

\begin{definition}
	Let
	$\mathscr{A} =[\mathscr{A}_{1},\ldots,\mathscr{A}_{p}]\in {\mathbb R}^{n_{1}\times ps \times n_3},$ where $  \mathscr{A}_{i} \; n_{1}\times s\times n_{3} , \, i =1,...,p$ 
	and let $	\mathscr{B} =[\mathscr{B}_{1},\ldots,\mathscr{B}_{l}]\in {\mathbb R}^{n_{1}\times \ell s \times n_3}$ with $  \mathscr{B}_{j}; n_{1}\times s\times n_{3}, \, j =1,...\ell$. Then,  the product $\mathscr{A}^{T} \diamondsuit  \mathcal{B} $ is the size  matrix $p \times  \ell  $ given by : 
	$$ (\mathscr{A}^{T} \diamondsuit  \mathcal{B})_{i,j}  =   \langle \mathscr {A}_i  ,\mathscr{B}_{j} \rangle  \;\;.   
	$$ 
\end{definition}
\section{Global tensor T-Arnoldi and global tensor T-Golub-Kahan}

\subsection{The tensor T-global GMRES }
Consider now the following tensor linear  system of equations
\begin{equation}\label{syslintens}
\mathscr{A}\star \mathscr{X}=\mathscr{C},
\end{equation}  
where $\mathscr{A}\in \mathbb{R}^{n\times n \times p}$, $\mathscr{C}$ and $ \mathscr{X}\in \mathbb{R}^{n\times s \times p}$. \\ 
\noindent We introduce the  tensor Krylov subspace $\mathcal{\mathscr{TK}}_m(\mathscr{A},\mathscr{V} )$ associated to the T-product, defined for  the pair $(\mathscr{A},\mathscr{V})$   as follows
\begin{equation}
\label{tr3}
\mathcal{\mathscr{TK}}_m(\mathscr{A},\mathscr{V} )= {\rm Tspan}\{ \mathscr{V}, \mathscr{A} \star\mathscr{V},\ldots,\mathscr{A}^{m-1}\star\mathscr{V} \}\\
=\left\lbrace \mathscr{Z} \in \mathbb{R}^{n\times s \times n_3}, \mathscr{Z}= \sum_{i=1}^m \alpha_{i} \left(
\mathscr{A}^{i-1}\star\mathscr{V}\right) \right\rbrace 
\end{equation}
where $\alpha_{i}\in \mathbb{R} $,  $\mathscr{A}^{i-1}\star\mathscr{V}=\mathscr{A}^{i-2}\star\mathscr{A}\star\mathscr{V}$, for $i=2,\ldots,m$ and $\mathscr{A}^{0}$ is the identity tensor. We can now  give a new version of the Tensor T-global Arnoldi algorithm.

\begin{algorithm}[H]
	\caption{Tensor T-global Arnoldi} \label{TGA}
	\begin{enumerate}
		\item 	{\bf Input.} $\mathscr{A}\in \mathbb{R}^{n\times n \times p}$, $\mathscr{V}\in \mathbb{R}^{n\times s \times p}$ and and the positive integer $m$.
		\item Set $\beta=\|\mathscr{V}\|_F$, $\mathscr{V}_{1} =   \dfrac{\mathscr{V}}{  \beta}$
		\item For $j=1,\ldots,m$
		\begin{enumerate}
			\item $\mathscr{W}=  \mathscr{A}\star   \mathscr{V}_j$
			\item for $i=1,\ldots,j$
			\begin{enumerate}
				\item $h_{i,j}=\langle \mathscr{V}_i, \mathscr{W} \rangle$
				\item $\mathscr{W}=\mathscr{W}-h_{i,j}\;\mathscr{V}_i$	
			\end{enumerate}	
			\item End for
			\item   $h_{j+1,j}=\Vert \mathscr {W} \Vert_F$. If $h_{j+1,j}=0$, stop; else
			\item $\mathcal {V}_{j+1}=\mathscr {W}/h_{j+1,j}$.

		\end{enumerate}
		\item End
	\end{enumerate}
\end{algorithm}

\begin{proposition}
	Assume that m steps of Algorithm (\ref{TGA}) have been run. Then, the tensors  $\mathscr{V}_{1},\ldots,\mathscr{V}_{m}$, form an orthonormal basis of the tensor global Krylov  subspace $\mathscr{TK}^{g}_{m}(\mathscr{A},\mathscr{V})$.
\end{proposition}

\medskip
\begin{proof}
	\noindent This can be shown easily by induction on $m$.
\end{proof}

\medskip
\noindent Let $\mathbb{V}_{m}  $ be  the $(n\times (sm)\times p)$ tensor with frontal slices $\mathscr{V}_{1},\ldots,\mathscr{V}_{m}$ and let $ {\widetilde{H}}_{m}$ be  the $(m+1)\times m  $ upper  Hesenberg matrix whose elements are the $h_{i,j}$'s defined by Algorithm \ref{TGA}.    Let  $ {H}_{m}$ be the matrix obtained from $\widetilde{ { H}}_{m}$ by deleting its last row;  $H_{.,j}$ will denote the $j$-th column of the matrix  $H_m$ and $\mathscr{A}\star\mathbb{V}_{m}  $ is  the $(n\times (sm)\times p)$ tensor with frontal slices  $\mathscr{A}\star\mathscr{V}_{1},\ldots,\mathscr{A}\star\mathscr{V}_{m}$ respectively given by 
\begin{equation}
\label{ev1}
\mathbb{V}_{m}:=\left[  \mathscr{V}_{1},\ldots,\mathscr{V}_{m}\right], \;\;\; {\rm and}\;\;\; \mathscr{A}\star\mathbb{V}_{m}:=[\mathscr{A}\star\mathscr{V}_{1},\ldots,\mathscr{A}\star\mathscr{V}_{m}].
\end{equation}
We introduce the product  $\circledast$ defined by: $$\mathbb{V}_{m}\circledast y=\sum_{j=1}^{m} {y}_{j}\mathscr{  {V}}_{j},\; y= (y_1,\ldots,y_m)^T\in \mathbb{R}^m.$$
We set the following notation:
\begin{equation*}
\mathbb{V}_{m}\circledast {{ {H}}_{m}}=\left[   \mathbb{V}_m\circledast H_{.,1} ,\ldots,\mathscr{V}_{m}\circledast H_{.,m} \right].
\end{equation*}
Then, it is easy to see that $\forall$   $u \,\text{and}\,  v$  $\in \mathbb{R}^{m}$, we have 
\begin{equation}
\mathbb{V}_{m}\circledast (u+v)=\mathbb{V}_{m}\circledast u + \mathbb{V}_{m}\circledast v \quad \text{and}\quad (\mathbb{V}_{m}\circledast H_m)\circledast u=\mathbb{V}_{m}\circledast(H_m\;u).
\end{equation}
With these notations, we can show the following result (proposition) that will be useful later on.
\medskip

\begin{proposition}\label{normfrobnorm2}
	Let $\mathbb{V}_{m}$ be the tensor defined by  $\left[ \mathscr{V}_{1},\ldots,\mathscr{V}_{m}\right]$ where $\mathscr{V}_{i}\in \mathbb{R}^{n\times s\times p} $ are defined by the Tensor T-global Arnoldi algorithm. Then, we have
	\begin{equation}
	\|\mathbb{V}_{m}\circledast y\|_F=\|y\|_2, \; \forall y= (y_1,\ldots,y_m)^T\in \mathbb{R}^m .
	\end{equation}
\end{proposition}

\begin{proof}
	From the definition of  the  product  $\circledast$,  we have $\sum_{j=1}^{m} {y}_{j}\mathscr{  {V}}_{j}=\mathbb{V}_{m}\circledast y$. Therefore,  
	$$\|\mathbb{V}_{m}\circledast y\|_F^2= \left< \sum_{j=1}^{m} {y}_{j}\mathscr{  {V}}_{j},\sum_{j=1}^{m} {y}_{j}\mathscr{  {V}}_{j} \right>_F.$$
	But, since the tensors $ \mathscr{  {V}}_{i}$'s are orthonormal, it follows that
	$$\|\mathbb{V}_{m}\circledast y\|_F^2= \sum_{j=1}^{m} {y}_{j}^2={\|y\|_2^2},$$
	which shows the result.
	
\end{proof}

\noindent With the above notations, we can easily prove the  results of the following proposition :

\begin{proposition}\label{T-GlobalArnolproposit}
	Suppose that m steps of Algorithm \ref{TGA} have been   run. Then, the following statements hold:
	\begin{eqnarray}
	\mathscr{A}\star\mathbb{V}_{m}&=&\mathbb{V}_{m}\circledast {{ {H}}_{m}} +  h_{m+1,m}\left[  \mathscr{O}_{n\times s\times p},\ldots,\mathscr{O}_{n\times s\times p},\mathscr{V}_{m+1}\right],\\
	\mathscr{A}\star\mathbb{V}_{m}&=&\mathbb{V}_{m+1}   \circledast \widetilde{ { H}}_{m}, \\
	\mathbb{V}_{m}^{T}\diamondsuit\mathscr{A}\star\mathbb{V}_{m}&=& {H}_{m}, \\	
	\mathbb{V}_{m+1}^{T}\diamondsuit  \mathscr{A}\star\mathbb{V}_{m}&=&\widetilde{ { H}}_{m},\\
	\mathbb{V}_{m}^{T} \diamondsuit\mathbb{V}_m&=& {I}_{ m  },
	\end{eqnarray}
	where ${I}_{ m  }$ the identity matrix and $\mathscr{O}$ is the tensor having all its entries equal to zero.
\end{proposition}

\medskip
\begin{proof}
	From Algorithm \ref{TGA}, we have $ \mathscr{A}\star\mathscr{V}_{j}=\sum_{i =1}^{j+1}h_{i,j}  \mathscr{V}_{i}$.
	Using the fact that  $\mathscr{A}\star\mathbb{V}_{m}=\left[\mathscr{A}\star\mathscr{V}_{1},\ldots,\mathscr{A}\star\mathscr{V}_{m}\right]$,  
	the $j$-th frontal slice    of $\mathscr{A}\star\mathbb{V}_{m}$ is given by 
	\begin{align*}
	(\mathscr{A}\star\mathbb{V}_{m})_j=\mathscr{A}\star\mathscr{V}_{j}&=\sum_{i =1}^{j+1}h_{i,j} \mathscr{V}_{i}.	 
	\end{align*}
	Furthermore, from the definition of the $\circledast$ product, we have 
	\begin{align*}
	(\mathbb{V}_{m+1}   \circledast \widetilde{ { H}}_{m})_j&=\mathbb{V}_{m+1}\circledast H_{.,j},\\
	&= \sum_{i =1}^{j+1}h_{i,j} \mathscr{V}_{i},
	\end{align*}
	which proves the first two relations. The other relations follow from the definition of T-diamond product
\end{proof} 

\medskip
\noindent In the sequel, we develop the tensor T-global  GMRES algorithm  for solving the problem \eqref{syslintens}. It could be considered as generalization of the well known global GMERS algorithm \cite{jbilou1}. 
Let $\mathscr{  {X}}_{0}\in \mathbb{R}^{n\times s\times p}$ be an arbitrary initial guess with   the corresponding  residual
$\mathscr{R}_0=\mathscr{C}-\mathscr{A}\star \mathscr{X}_0$.    The aim of tensor  T-global GMRES method is to find and approximate solution  $\mathscr{X}_{m}$ approximating the exact solution $\mathscr{X}^*$ of \eqref{syslintens}  such that 
\begin{equation}
\label{gmres1}
\mathscr{X}_{m}-\mathscr{X}_{0}\in \mathscr{TK}^{g}_{m}(\mathscr{A},\mathscr{R}_0), 
\end{equation} 
with the classical  minimization property 
\begin{equation}
\label{gmres2} 
\Vert \mathscr{R}_{m}\Vert_F = \displaystyle \min_{ \mathscr{X} \in \mathscr{X}_{0} + \mathscr{TK}^{g}_{m}(\mathscr{A},\mathscr{R}_0)}\left\lbrace \|\mathscr{C}-\mathscr{A}\star \mathscr{X}\|_F
\right\rbrace.
\end{equation} 

\noindent Let   $\mathscr{X}_{m}=\mathscr{X}_{0}+\mathbb{V}_{m}\circledast y  $ with $ {y} \in \mathbb{R}^m $,   be the approximate solution satisfying \eqref{gmres1}. Then, 
\begin{align*}
\mathscr{R}_m=&\mathscr{C}-\mathscr{A}\star\mathscr{X}_{m},\\
=& \mathscr{C}-\mathscr{A}\star\left(\mathscr{X}_{0}+ \mathbb{V}_{m}\circledast y\right), \\
=& \mathscr{C}-\mathscr{A}\star\mathscr{X}_{0}-\mathscr{A} \star(\mathbb{V}_{m}\circledast y),\\
=&\mathscr{R}_{0}-\left(\mathscr{A} \star\mathbb{V}_{m}\right)\circledast y. 
\end{align*}
It follows then that 
\begin{align*}
\|\mathscr{R}_m \|_{F}&=  \displaystyle \min_{ y\in \mathbb{R}^{m }}
\|\mathscr{R}_{0}-(\mathscr{A}\star\mathbb{V}_{m})\circledast y\|_F,
\end{align*}
where $\mathscr{A}\star\mathbb{V}_{m}:=[\mathscr{A}\star\mathscr{V}_{1},\ldots,\mathscr{A}\star\mathscr{V}_{m}]
$ is the $(n\times sm\times p)$ tensor defined earlier.\\
\noindent Using  Propositions \ref{normfrobnorm2} and the fact that      $\mathscr{R}_{0}=\|\mathscr{R}_{0}\|_F \mathscr{V}_1 $  with  $\mathscr{V}_1 = \mathscr{V}_{m+1}\circledast e_{1}$, where $e_{1}$  the first  canonical basis vector in $\mathbb{R}^{m+1}$,  we get 
\begin{align*}
\|\mathscr{R}_{0}-(\mathscr{A}\star\mathbb{V}_{m})\circledast y\|_F&=\| \mathscr{R}_{0}- (\mathbb{V}_{m+1}   \circledast \widetilde{ { H}}_{m}) \circledast y  \|_F,\\
&=\|\|\mathscr{R}_{0}\|_F (\mathbb{V}_{m+1}\circledast e_1)- (\mathbb{V}_{m+1}   \circledast \widetilde{ { H}}_{m}) \circledast y \|_F,\\
&=\| \mathbb{V}_{m+1}\circledast (||\mathscr{R}_{0}\|_F e_1-\widetilde{ { H}}_{m}  y) \|_F ,\\
&=\|\; \|\mathscr{R}_{0}\|_F\; e_1-\widetilde{ { H}}_{m}   y \|_2.\\
\end{align*}
Finally, we obtain  
\begin{equation}\label{solutdegmresxm}
\mathscr{X}_{m}=\mathscr{X}_{0}+ \mathbb{V}_{m} \circledast y, 
\end{equation}
where,
\begin{equation}\label{Gmressol}
y=  \text{arg } \min_{{{y}}\in \mathbb{R}^{m }}||\; ||\mathscr{R}_{0}||_F\; e_1-\widetilde{ { H}}_{m}  y) ||_2.
\end{equation}
\subsection{Tensor T-global Golub Kahan algorithm}

\noindent Instead of using the tensor T-global Arnoldi to generate a basis for the projection subspace, we can define T-version of the tensor global Lanczos process. Here, we will use the tensor Golub Kahan algorithm related to the T-product. We notice here that we already defined in \cite{Elguide} another version of the tensor Golub Kahan algorithm by using the $m$-mode or the Einstein products with applications to color image restoration.\\
Let $\mathscr{A} \in \mathbb{R}^{n\times \ell\times p}$  be a tensor and  let  $\mathscr{U}  \in \mathbb{R}^{\ell\times  s \times p}$ and  $\mathscr{V}  \in \mathbb{R}^{n\times  s \times p}$  two other tensors. Then, the Tensor T-global Golub Kahan bidiagonalization algorithm (associated to the T-product)  is defined as follows

\begin{algorithm}[h!]
	\caption{The Tensor T-global Golub Kahan algorithm}\label{TG-GK}
	\begin{enumerate}
		\item {\bf Input.} The tensors $\mathscr {A}$, $\mathscr{V}$, and $\mathscr {U}$ and an integer $m$.
		\item 	Set $\beta_1= \Vert \mathscr{V}\Vert_F$, $\alpha_1= \Vert \mathscr {U}  \Vert_F$, $\mathscr {V}_1=\mathscr {V}/\beta_1$ and 
		$\mathscr {U}_1=\mathscr {U}/\alpha_1$.
		\item for $j=2,\ldots,m$
		\begin{enumerate}
			\item $\widetilde {\mathscr {V}}= \mathscr {A} \star \mathscr {U}_{j-1} -\alpha_{j-1}\mathscr {V}_{j-1}$
			\item $\beta_j=\Vert \widetilde {\mathscr {V}}\Vert_F$ if $\beta_j=0$ stop, else
			\item $\mathscr {V}_j=\widetilde {\mathscr {V}}/\beta_j$
			\item $\widetilde {\mathscr {U}}=\mathscr {A}^T \star \mathscr {V}_j-\beta_j \mathscr{U} _{j-1}$
			\item $\alpha_j=\Vert \widetilde {\mathscr {U}} \Vert_F$
			\item if $\alpha_j=0$ stop, else
			\item $\mathscr {U}_j=\widetilde {\mathscr {U}}/\alpha_j$
		\end{enumerate}
	\end{enumerate}
	
\end{algorithm}

\medskip 

\noindent Let $\widetilde{C}_m$ be the upper bidiagonal $((m+1) \times m  )$ matrix 
$$ \widetilde{ {   {C}}}_m=\left[ \begin{array}{*{20}{c}}
{{{\alpha}_1  }}&{{ }}& &   \\
{\beta}_{2}&{{{\alpha}_2}}&\ddots& \\
&\ddots&\ddots& \\
&   & {\beta}_{m}  &  {\alpha}_{m}\\
&    &       &    {\beta}_{m+1}
\end{array}  \right] 
$$
and let $ {{{C}}}_m$ be the $(m \times m )$ matrix obtain   by deleting  the last row of  $\widetilde{{   {C}}}_m$. We denote by  $C_{.,j}$ will denote the $j$-th column of the matrix  $C_m$. Let  $\mathbb{U}_{m}  $ and $\mathscr{A}\star\mathbb{U}_{m}  $ be the $(\ell\times (sm)\times p)$ and   $(n\times (sm)\times p)$ tensors with frontal slices $\mathscr{U}_{1},\ldots,\mathscr{U}_{m}$ and  $\mathscr{A}\star\mathscr{U}_{1},\ldots,\mathscr{A}\star\mathscr{U}_{m}$, respectively, and let  $\mathbb{V}_{m}  $ and $\mathscr{A}^T\star\mathbb{V}_{m}  $ be the $(n\times (sm)\times p)$ and $(\ell\times (sm)\times p)$  tensors with frontal slices $\mathscr{V}_{1},\ldots,\mathscr{V}_{m}$ and  $\mathscr{A}^T\star\mathscr{V}_{1},\ldots,\mathscr{A}^T\star\mathscr{V}_{m}$, respectively. We set  
\begin{align}
\label{ev12}
\mathbb{U}_{m}:&=\left[  \mathscr{U}_{1},\ldots,\mathscr{U}_{m}\right], \;\;\; {\rm and}\;\;\; \mathscr{A}\star\mathbb{U}_{m}:=[\mathscr{A}\star\mathscr{U}_{1},\ldots,\mathscr{A}\star\mathscr{U}_{m}],\\
\mathbb{V}_{m}:&=\left[  \mathscr{V}_{1},\ldots,\mathscr{V}_{m}\right], \;\;\; {\rm and} \;\;\; \mathscr{A}^T\star\mathbb{V}_{m}:=[\mathscr{A}^T\star\mathscr{V}_{1},\ldots,\mathscr{A}^T\star\mathscr{V}_{m}].
\end{align}
\noindent  Then, the following proposition can be established\\

\begin{proposition}\label{proptggkb} 
	The tensors produced by the tensor T-global Golub-Kahan algorithm satisfy the following relations
	\begin{eqnarray} \label{equa20}
	\mathcal {A} \star \mathbb{U}_m& = &\mathbb{V}_{m+1} \circledast {\widetilde { {   {C}}}}_m   ,    \\
	& = &\mathbb{V}_m\circledast{  { {   {C}}}}_m  + {\beta}_{m+1}  \left[  \mathscr{O}_{n\times s\times p},\ldots,\mathscr{O}_{n\times s\times p},\mathscr{V}_{m+1}\right], \\
	\mathscr{A}^{T}\star\mathbb{V}_{m}& = &\mathbb{U}_m  \circledast {\widetilde { {   {C}}}}_m^T .  
	\end{eqnarray}
\end{proposition} 

\begin{proof}
	Using $\mathscr{A}\star\mathbb{U}_{m}=[\mathscr{A}\star\mathscr{U}_{1},\ldots,\mathscr{A}\star\mathscr{U}_{m}] \in  \mathbb{R}^{n \times (sm)\times n_{3}} $ , the ($j-1$)-th lateral slice    of $(\mathscr{A}\star\mathbb{U}_{m})$ is given by   $$ (\mathcal {A} \star \mathbb{U}_m)_{j-1} =\mathscr {A} \star \mathscr {U}_{j-1}= {\alpha }_{j-1} \mathscr {V}_{j-1}+{\beta }_j \mathcal {V}_{j}.$$
	Furthermore, from the definition of the $\circledast$ product, we have
	\begin{align*}
	(\mathbb{V}_{m+1}   \circledast \widetilde{ { C}}_{m})_{j-1}&=\mathbb{V}_{m+1}\circledast C_{.,j-1},\\
	&= \sum_{i =1}^{j+1}c_{i,j-1} \mathscr{V}_{i},\\
	&={\alpha }_{j-1} \mathscr {V}_{j-1}+{\beta }_j \mathcal {V}_{j}
	\end{align*}
	and for $j=m$,  $\mathbb{U}_{m} \circledast \mathscr{C}_{.,m}= \mathscr{A}\star \mathscr{U}_{m}+{\beta}_{m+1} \mathscr{V}_{m+1} $ and the result follows.\\
	To derive (4.5) , one may first
	notice that from Algorithm \ref{TG-GK}, we  have  $$(\mathcal {A}^T \star \mathbb{V}_m)_{j} =\mathscr {A} ^T\star \mathscr {V}_{j}= {\alpha }_{j} \mathscr {U}_{j}+{\beta }_j \mathcal {U}_{j-1}.$$
	Considering now the $j$-th frontal slice of the right-hand side
	of (4.5), the assertion can be easily deduced .
\end{proof}

\medskip
\begin{proposition}
	Let $\mathscr{  {X}}_{m}=  \mathscr{  {X}}_{0}+\mathbb{U}_{m}\circledast y   \in \mathbb{R}^{\ell\times s\times p}$ with $  {  {y}}\in \mathbb{R}^{m },$ where $\mathbb{U}_{m}$ is obtained from Algorithm \ref{TG-GK},  be an  approximation of   (\ref{syslintens}). Then,  we have
	\begin{equation}
	\|\mathscr{C}-\mathscr{A}\star \mathscr{X}_m\|_F=  \| \beta_{1}e_1- { \widetilde{C}}_m y \|_2,
	\end{equation}
	where  $\beta_{1}=\|\mathscr{C}\|_F$.
\end{proposition}
\medskip
\begin{proof}
	Using  representation (\ref{equa20})  and the fact that $\mathscr{C}=\mathbb{V}_{m+1}\circledast (\beta_{1}e_1)$  with  $\beta_{1}=\|\mathscr{C}\|_F$,  we get 
	\begin{align*}
	\|\mathscr{C}-\mathscr{A}\star \mathscr{X}_m\|_F&=\|\mathbb{V}_{m+1}\circledast (\beta_{1}e_1)  - (\mathbb{V}_{m+1} \circledast { \widetilde{C}}_m)\circledast y  |\|_{F},\\
	&=|||\mathbb{V}_{m+1}\circledast (\beta_{1}e_1-{\widetilde{C}}_m y   )||_F,\\
	&=||\beta_{1}e_1- {\widetilde{C}}_m y     ||_2.
	\end{align*}
\end{proof}
\section{Application to discrete-ill posed tensor problems}
We consider the following discrete ill-posed tensor equation
\begin{equation}\label{tr1}
\mathscr{A} \star \mathscr{X}= \mathscr{C},\quad \mathscr{C}=\widehat{ \mathscr C}+  \mathscr{N},
\end{equation}
where $\mathscr{A} \in {\R}^{n \times n \times s}$, $\mathscr{X}$,  $ \mathscr{N}$ (additive noise) and $\mathscr{C}$ are tensors in ${\R}^{n \times s \times p}$. \\
In color image processing, $p=3$, $\mathscr{A}$ represents the blurring tensor, $\mathscr{C}$ the blurry and noisy observed  image,  $\mathscr{X}$ is the image that we would like to restore  and  $\mathscr{N}$ is an unknown additive noise. Therefore, to stabilize  the recovered image, regularization techniques are  needed. There are several techniques to regularize the linear inverse problem given by
equation (\ref{tr1}); for the matrix case, see for example, \cite{belguide,reichel2,golubwahba,hansen1}. All of these techniques stabilize the restoration process by adding a regularization term, depending on some  priori knowledge of  the unknown image. One of the most regularization method is due to Tikhonov and is given as follows 
\begin{equation}\label{tr2}
\underset{\mathscr{X}}{\text{min}}\{\|\mathscr{A} \star \mathscr{X} - \mathscr{C}  \|_F^2+\mu \|\mathscr{X}\|_F^2\}.
\end{equation}
As  problem \eqref{tr1} is large,  Tikhonov regularization \eqref{tr2} may be very expensive to solve. One possibility is instead of regularizing the original problem, we apply the Tikhonov technique to the projected problem \eqref{Gmressol} which leads to the following problem
\begin{eqnarray}
y_{m,\mu}&=& \arg \min_{ y \in \mathbb{R}^{m }}\left( \|\mathscr{R}_0\| e_1-\widetilde{ { H}}_{m}  y   \|_ 2 +\mu \| y \|_2
\right),\\
&=&\arg \min_{ y \in \mathbb{R}^{m }}\left  \Vert \left ( \begin{array}{ll}
\widetilde H_m\\
\mu I_m
\end{array}\right ) y - \left ( \begin{array}{ll}
\beta e_1\label{min}\\
0
\end{array}\right ) 
\right \Vert_2.
\end{eqnarray}
The minimizer $y_{m,\mu}$ can also be computed as the solution of the following normal equations associated with  \eqref{min}
\begin{equation}{\label{tikho2}}
\widetilde H_{m,\mu} y=\widetilde H_m^T, \quad \widetilde H_{m,\mu}= (\widetilde H_m^T \widetilde H_m+ \mu^2 I_m).
\end{equation}
Note that since the Tikhonov problem \eqref{tikho2} is now a matrix one with small dimension as $m$ is generally small, the vector $y_{m,\mu}$, can thereby be inexpensively computed by some techniques such as the GCV method \cite{golubwahba} or the L-curve criterion \cite{hansen1,hansen2,reichel1,reichel2}. To choose the regularization parameter, we can use  the
generalized cross-validation (GCV) method
\cite{golubwahba,wahbagolub}. Now for the GCV method, the regularization
parameter is chosen by minimizing the following function
\begin{equation}\label{gcv}
GCV(\mu)=\frac{\|\widetilde H_m y_{m,\mu}-{\bf
		\beta e_1}\|_2^2}{[tr(I_m-\widetilde H_m  \widetilde H_{m,\mu}^{-1}\widetilde H_m^T)]^2}=\frac{\|(I_m-\widetilde H_m \widetilde H_{m,\mu}^{-1} \widetilde H_m^T){\beta e_1}\|_2^2}{[tr(I_m-H_m H_{m,\mu}^{-1} \widetilde H_m^T)]^2}.
\end{equation}

To minimize (\ref{gcv}), we take advantage of the the SVD decomposition of the low dimensional matrix $\widetilde H_m$ to obtain a more simple and computable expression of $GCV(\mu)$. Consider the SVD decomposition of $\widetilde H_m=U\Sigma V^T$. Then, the GCV is now expressed as (see \cite{wahbagolub})

\begin{equation}
\label{gcv2}
GCV(\mu)=\frac{\displaystyle
	\sum_{i=1}^m\left(\frac{\widetilde
		g_i}{\sigma_i^2+\mu^2}\right)^2}{\displaystyle\left(\sum_{i=1}^m
	\frac{1}{\sigma_i^2+\mu^2}\right)^2},
\end{equation}
where $\sigma_i$ is the $i$th singular value of the matrix
$\widetilde H_m$ and $\widetilde g= \beta_1 U^T e_1$.

In  terms of  practical implementations, it's more convenient to introduce a restarted version of the tensor Global GMRES. This strategy is essentially based on restarting the tensor T-global Arnoldi algorithm. Therefore, at each
restart, the initial guess $\mathscr{X}_0$ and the regularization parameter $\mu$
are updated employing the last values computed when the the number of inner iterations required is fulfilled. We note that as the number  outer iterations increases it is possible to compute the $m$th residual without having to compute extra T-products. This is described in the following proposition.
\begin{proposition}
	At step $m$, the residual $\mathscr{R}_{m}=\mathscr{C}-\mathscr{A}\ast \mathscr{X}_{m}$ produced by the tensor Global GMRES method for tensor equation (\ref{eq1}) has the following expression
	\begin{equation}\label{resex}
	\mathscr{R}_m=\mathbb{V}_{m+1} \circledast\left(\gamma_{m+1}Q_me_{m+1}\right),
	\end{equation}
	
	where $Q_m$ is the unitary matrix obtained from the QR decomposition of the upper Hessenberg matrix $\widetilde{H}_{m}$ and $\gamma_{m+1}$ is the last component of the vector $\left\|\mathscr{R}_{0}\right\|_{F} Q_{m}^{\mathrm{T}} e_{1}$ and $e_{m+1}=(0,0, \ldots, 1)^{\mathrm{T}} \in \mathbb{R}^{m+1}.\\$
	Furthermore,
	\begin{equation}\label{resnrm}
	\left\|\mathscr{R}_{m}\right\|_{F}=\left|\gamma_{m+1}\right|.
	\end{equation}
	
\end{proposition}
\begin{proof}
	At step $m$, the residual $\mathscr{R}_m$ can be expressed as
	$$\mathscr{R}_m=\mathbb{V}_{m+1} \circledast\left(\beta e_{1}-\widetilde{H}_{m} y_{m}\right),$$
	by considering the QR decomposition $\widetilde{H}_{m}=Q_{m}\widetilde{U}_m$ of the $(m + 1) \times m$ matrix $\widetilde{H}_{m}$, we get
	$$\mathscr{R}_m=\left(\mathbb{V}_{m+1} \circledast Q_m\right)\circledast\left(\beta Q_m^T e_{1}-\widetilde{U}_{m} y_{m}\right).$$
	Since $y$ solves problem (\ref{Gmressol}), it follows that
	$$\mathscr{R}_m=\mathbb{V}_{m+1} \circledast\left(\gamma_{m+1}Q_me_{m+1}\right),$$
	where $\gamma_{m+1}$ is the last component of the vector $\beta Q_{m}^T e_{1}.$ Therefore,
	\begin{eqnarray*}
		\left\|\mathscr{R}_{m}\right\|_{F}&=&\left\|\mathbb{V}_{m+1} \circledast\left(\gamma_{m+1}Q_me_{m+1}\right)\right\|_F,\\
		&=&\left\|\gamma_{m+1}Q_me_{m+1}\right\|_2,\\
		&=&\left|\gamma_{m+1}\right|,
	\end{eqnarray*}
	which shows the results. 
\end{proof}

\noindent The tensor T-global GMRES method is summarized in the following algorithm

\begin{algorithm}[h!]
	\caption{Implementation of Tensor T-global GMRES(m)}\label{TG-GMRES(m)}
	\begin{enumerate}
		\item 	{\bf Input.} $\mathscr{A}\in \mathbb{R}^{n\times n \times n_3}$, $\mathscr{V},\mathscr{B},\mathscr{X}_{0}\in \mathbb{R}^{n\times s \times n_3}$, the maximum number of iteration  $\text{Iter}_{\text{max}} $ and a tolerance $tol>0$ .
		\item 	{\bf Output.} $  \mathscr{X}_{m}\in \mathbb{R}^{n\times s\times n_3}$ approximate  solution of the system  (\ref{syslintens}).
		\item $k=1,\ldots,\text{Iter}_{\text{max}}$
		\begin{enumerate}
			\item  Compute $\mathscr{R}_{0}=\mathscr{C}-\mathscr{A}\star\mathscr{X}_{0} $.
			\item  Apply Algorithm \ref{TGA} to compute  $\mathbb{V}_{m}$ and  ${ \widetilde{H}}_m$ .
			%
			\item  Determine $\mu_{k}$  as the parameter minimizing the GCV function  given by (\ref{gcv2})
			\item  Compute the regularized solution $y_{m_k,\mu}$ of the problem \eqref{min}. 
			\item Compute the approximate solution  $\mathscr{X}_{m}=\mathscr{X}_{0}+ \mathbb{V}_{m} \circledast y_{m,\mu_k}  
			$
		\end{enumerate}
		\item If $\|\mathscr{R}_{m}\|_F<tol$, stop, else
		\item Set    $\mathscr{X}_{0}=\mathscr{X}_{m}$ and go to 3-a.
		\item End
	\end{enumerate}
\end{algorithm}

\noindent We turn now to the tensor T-global Golub Kahan approach for the solving the Tikhonov regularization of the problem (\ref{eq1}). Here, we apply the following Tikhonov regularization approach and solve the new problem

\begin{equation}\label{tikho3}
\underset{\mathscr{X}}{\text{min}}\{\|\mathscr{A} \star \mathscr{X} - \mathscr{C}  \|_F^2+\mu^{-1} \|\mathscr{X}\|_F^2\}.
\end{equation}
The  use of  $\mu^{-1} $
in (\ref{tikho3}) instead of $\mu$  will be justified below. In the what follows, we briefly review the discrepancy principle approach to determine a suitable regularization parameter, given an approximation of the norm of the additive error. We then assume that a bound $\varepsilon$ for $\|\mathscr{N}\|_F$ is available. This priori information suggests that $\mu$ has to be determined  as soon as
\begin{equation}\label{discrepancy}
\phi(\mu)\leq\eta\epsilon,
\end{equation}
where $\phi(\mu)=\|\mathscr{A} \star \mathscr{X} - \mathscr{C}  \|_F^2$ and  $ \eta\gtrapprox 1$ is refereed to  as the safety factor for the discrepancy principle. A zero-finding method can be used to solve (\ref{discrepancy}) in order to find a suitable regularization parameter which also implies that $\phi(\mu)$ has to be evaluated for several $\mu$-values. When the tensor  $\mathscr{A}$ is of moderate size, the quantity $\phi(\mu)$ can be easily evaluated. This evaluation becomes expensive when the matrix $\mathscr{A}$ is large, which means that its evaluation by a  zero-finding method can be very difficult and computationally expensive.  We will approximate $\phi$ to be able to determine an estimate of  $\|\mathscr{A} \star \mathscr{X} - \mathscr{C}  \|_F^2$.  Our approximation is obtained by using  T-global Golub-Kahan bidiagonalization (T-GGKB) and  Gauss-type quadrature rules. This connection provides approximations of moderate sizes to the quantity  $\phi$, and therefore gives a solution method to inexpensively solve (\ref{discrepancy}) by evaluating these small quantities that can successfully and inexpensively be employed to compute $\mu$ as well as defining a stopping criterion for  the T-GGKB iterations;  see \cite{belguide, belguide2} for discussion on this method.\\
Introduce the functions (of $\mu$)
\begin{eqnarray}\label{Gkfmu}
\mathcal{G}_m f_\mu&=&\|\mathscr{C}\|_F^2 e_1^T(\mu C_m C_m^T+I_m)^{-2}e_1,\\
{\mathcal R}_{m+1}f_\mu&=&\|\mathcal{C}\|_F^2 e_1^T(\mu \widetilde{C}_m\widetilde{C}_m^T+I_{m+1})^{-2}e_1;
\end{eqnarray}
The quantities   $\mathcal{G}_m f$ and ${\mathcal R}_{m+1}f_\mu$ are refereed to as   Gauss and Gauss-Radau quadrature rules, respectively, and can be obtained after $m$ steps of T-GGKB (Algorithm \ref{TG-GK}) applied to tensor $\mathscr{A}$ with initial tensor $\mathscr{C}$. These quantities  approximate $\phi(\mu)$ as follows
\begin{equation}
\mathcal{G}_m f_\mu\leq\phi(\mu)\leq{\mathcal R}_{m+1}f_\mu.
\end{equation}
Similarly to the approaches proposed  in \cite{belguide, belguide2}, we therefore instead  solve for $\mu$ the low dimensional nonlinear equation
\begin{equation}\label{lin22}
{\mathcal G}_m f_\mu=\epsilon^2.
\end{equation}
We apply the Newton's method to solve (\ref{lin22}) that requires repeated evaluation of the function ${\mathcal G}_m f_\mu$ and its derivative, which are inexpensive computations for small $m$.\\
We now comment on the use of  $\mu$ in (\ref{tikho3}) instead of $1 / \mu,$ implies that the left-hand side of (\ref{discrepancy}) is a decreasing convex function of $\mu .$ Therefore, there is a unique solution, denoted by $\mu_{\varepsilon},$ of
$$
\phi(\mu)=\varepsilon^{2}
$$
for almost all values of $\varepsilon>0$ of practical interest and therefore also of (\ref{lin22}) for $m$ sufficiently large; see \cite{belguide, belguide2} for analyses.  We accept $\mu_m$ that solve (\ref{discrepancy}) as an approximation of $\mu$, whenever we have
\begin{equation}\label{upperbd}
{\mathcal R}_{m+1}f_{\mu}\leq\eta^2\epsilon^2. 
\end{equation}
If (\ref{upperbd}) does not hold for $\mu_m$, we carry out one more GGKB steps, replacing $m$ by $m+1$ and solve the nonlinear equation
\begin{equation}\label{}
{\mathcal G}_{m+1}f_\mu=\epsilon^2;
\end{equation}
see \cite{belguide, belguide2} for more details. Assume now that (\ref{upperbd}) holds for some $\mu_m$. The corresponding regularized solution is then computed by
\begin{equation}\label{Xkmu}
\mathscr {X}_{m,\mu_m}=\mathbb{U}_m \circledast y_{m,\mu_m}, 
\end{equation}
where $y_{m,\mu_m}$ solves 
\begin{equation}\label{normeq2}
(\widetilde{C}_m^T\widetilde{C}_m+\mu_m^{-1} I_m)y=\beta_1\widetilde{C}_m^Te_1,\qquad\beta_1=\|\mathscr{C}\|_F.
\end{equation}
It is also computed by solving the least-squares problem
\begin{equation}\label{leastsq}
\min_{y\in\mathbb{R}^m} \begin{Vmatrix}
\begin{bmatrix}
\mu_m^{1/2}\widetilde{C}_m\\
I_m
\end{bmatrix}
y-\beta_1\mu_m^{1/2}e_1 \end{Vmatrix}_2.
\end{equation}The following result shows an important property of the approximate solution (\ref{Xkmu}). We include a proof for completeness.

\begin{proposition}
	Let $\mu_{m}$ solve (\ref{lin22}) and let $y_{m,\mu_m}$ solve (\ref{leastsq}). Then the associated approximate solution (\ref{Xkmu}) of (\ref{tikho3}) satisfies
	$$
	\left\|\mathscr{A}\ast \mathscr{X}_{m,\mu_m}-\mathscr{C}\right\|_{F}^{2}=R_{m+1} f_{\mu_{m}}.
	$$
\end{proposition}
\begin{proof} The representation of Proposition \ref{proptggkb} show that
	$$
	\mathscr{A}\ast \mathscr{X}_{m,\mu_m}=(\mathscr{A}\ast \mathbb{U}_{m})\circledast y_{m,\mu_m}= \mathbb{V}_{m+1}\circledast(\widetilde{C}_m y_{m,\mu_m}). 
	$$
	
	Using the above expression gives
	
	$$
	\begin{aligned}
	\left\|\mathscr{A}\star \mathscr{X}_{m,\mu_m}-\mathscr{C}\right\|_{F}^{2}&=\left\|\mathbb{V}_{m+1}\circledast(\widetilde{C}_m y_{m,\mu_m})-\beta_1\mathscr{V}_1\right\|_{F}^{2},  \\
	&=\left\|\mathbb{V}_{m+1}\circledast(\widetilde{C}_m y_{m,\mu_m})-\mathbb{V}_{m+1}\circledast(\beta_1e_1)\right\|_{F}^{2}, \\
	&=\left\|\mathbb{V}_{m+1}\circledast\left(\widetilde{C}_m y_{m,\mu_m}-\beta_1e_1\right)\right\|_{F}^{2}, \\
	&=\left\|\widetilde{C}_\ell y_{m,\mu_m}-\beta_1e_1\right\|_{2}^{2}.
	\end{aligned}$$
	where we recall that $\beta_1=\|\mathcal{C}\|_{F}$. We now express $y_{m,\mu_m}$ with the aid of  (\ref{normeq2}) and apply the following  identity 
	$$I-A\left(A^{T} A+\mu^{-1} I\right)^{-1} A^{T}=\left(\mu A A^{T}+I\right)^{-1}$$
	with $A$ replaced by $\widetilde{C}_{m},$ to obtain
	$$
	\begin{aligned}
	\left\|\mathscr{A}\ast \mathscr{X}_{m,\mu_m}-\mathscr{C}\right\|_{F}^{2} &=\beta_1^{2}\left\|e_{1}-\widetilde{C}_{m}\left(\widetilde{C}_{m}^{T} \widetilde{C}_{m}+\mu_{m}^{-1} I_{m}\right)^{-1} \widetilde{C}_{m}^{T} e_{1}\right\|_{F}^{2}, \\
	&=\beta_1^{2} e_{1}^{T}\left(\mu_{m} \widetilde{C}_{m} \widetilde{C}_{m}^{T}+I_{m+1}\right)^{-2} e_{1}, \\
	&=R_{m+1} f_{\mu_{m}}.
	\end{aligned}
	$$
\end{proof}

\noindent  The following algorithm summarizes the main steps to compute a regularization parameter and a corresponding regularized solution of (\ref{eq1}), using  Tensor T-GGKB and quadrature rules method for Tikhonov regularization.

\begin{algorithm}[!h]
	\caption{ Tensor T-GGKB and quadrature rules method for Tikhonov regularization}\label{TG-GKB}
	\begin{enumerate}
		\item {\bf Input.} $\mathscr{A}\in \mathbb{R}^{n\times n \times n_3}$,  $\mathscr {C}$, $\eta\gtrapprox 1$  and  $\varepsilon$.
		\item {\bf Output.} T-GGKB steps $m$, $\mu_{m}$ and $X_{m,\mu_m}$.
		\item Determine the orthonormal bases $\mathbb{U}_{m+1}$ and $\mathbb{V}_{m}$ of tensors, and the bidiagonal $C_m$ and $\widetilde{C}_m$
		matrices  with Algorithm \ref{TG-GK}.
		\item Determine $\mu_{m}$ that satisfies (\ref{lin22}) with Newton's method.
		\item  Determine $y_{m,\mu_m}$ by solving  (\ref{leastsq}) and then compute $X_{m,\mu_m}$ by (\ref{Xkmu}).
	\end{enumerate}
\end{algorithm}
\section{Numerical results}
This section performs some numerical tests on the methods of Tensor T-Global GMRES(m) and Tensor T-Global Golub Kahan algorithm given by Algorithm \ref{TG-GMRES(m)} and Algorithm \ref{TG-GKB}, rspectively, 
when applied to the restoration of blurred and noisy color images and videos. For clarity, we only focus on the formulation of a tensor model (\ref{tr1}), describing the blurring that is taking place in the process of going from the exact to the blurred RGB image.  We recall that an RGB image is just multidimensional array of dimension $m\times n\times 3$ whose entries are the light intensity. Throughout this section, we assume that the the three channels
of the RGB image has the same dimensions, and we refer to it as $n\times n\times 3$ tensor. Let $\widehat{X}^{(1)}$, $\widehat{X}^{(2)}$, and $\widehat{X}^{(3)}$ be the $n\times n$ matrices that constitute the three channels of the original error-free color image $\widehat{\mathscr{X}}$, and $\widehat{C}^{(1)}$, $\widehat{C}^{(2)}$, and $\widehat{C}^{(3)}$ the $n\times n$ matrices associated with error-free blurred color image $\widehat{\mathscr{C}}$. Because of some
unique features in images, we seek an image restoration model that utilizes blur information,
exploiting the spatially invariant properties. Let us also consider that both cross-channel and within-channel blurring  take place in the blurring process of the original image. Let $\tt{vec}$ be the operator  that
transforms a matrix  to a vector  by stacking the columns of the matrix from
left to right. Then, the full blurring model is described by the following
form 
\begin{equation}\label{linmodel}
\left(\mathbf{A}_{\text {color }} \otimes \mathbf{A^{(1)}}\otimes\mathbf{A^{(2)}}\right) \widehat{\mathbf{x}}=\widehat{\mathbf{c}},
\end{equation}
where, 
$$ \widehat{\mathbf{c}}=\left[\begin{array}{c}
\tt{vec}\left(\widehat{\mathbf{C}}^{(1)}\right) \\
\tt{vec}\left(\widehat{\mathbf{C}}^{(2)}\right) \\
\tt{vec}\left(\widehat{\mathbf{C}}^{(3)}\right)
\end{array}\right], \quad \widehat{\mathbf{x}}=\left[\begin{array}{c}
\tt{vec}\left(\widehat{\mathbf{X}}^{(1)}\right) \\
\tt{vec}\left(\widehat{\mathbf{X}}^{(2)}\right) \\
\tt{vec}\left(\widehat{\mathbf{X}}^{(3)}\right)
\end{array}\right], $$
and 

$$\mathbf{A}_{\mathrm{color}}=\left[\begin{array}{ccc}
\alpha & \gamma& \beta \\
\beta & \alpha & \gamma \\
\gamma & \beta & \alpha
\end{array}\right]$$
$\mathbf{A}_{\text {color }}$ is the $3\times3$ matrix that models the  cross-channel blurring, where each row sums to one.  $\mathbf{A^{(1)}}\in\mathbb{R}^{n\times n}$ and $\mathbf{A^{(2)}}\in\mathbb{R}^{n\times n}$ define within-channel blurring and they model
the horizontal within blurring and the vertical  within blurring matrices, respectively; for more details see \cite{HNO}. The notation $\otimes$ denotes the Kronecker product of  matrices; i.e. the Kronecker product of a $n\times p$  matrix $A=(a_{ij})$  and a $(s\times q)$ matrix
$B=(b_{ij})$,   is defined as the $(ns)\times(pq)$ matrix $A \otimes B = (a_{ij}B)$. By exploiting  the circulant structure  of the cross-channel blurring matrix $\mathbf{A}_{\text {color }}$ and the operators unfold and fold, it can be easily shown that (\ref{linmodel}) can be written in the following tensor form
\begin{equation}
\mathscr {A}\star \widehat{\mathscr{X}}\star \mathscr{B}= \widehat{\mathscr{C}},
\end{equation}
where $\mathscr {A}$ is a 3-way tensor such that  $\mathscr {A}(:,:,1)=\alpha \mathbf{A^{(2)}}$, $\mathscr {A}(:,:,2)=\beta \mathbf{A^{(2)}}$ and $\mathscr {A}(:,:,3)=\gamma \mathbf{A^{(2)}}$ and  $\mathscr {B}$ is a 3-way tensor with $\mathscr {B}(:,:,1)=(\mathbf{A^{(1)}})^T$, $\mathscr {B}(:,:,2)=0$ and $\mathscr {B}(:,:,3)=0$. To test the performance of algorithms, the within blurring matrices $A^{(i)}$  have the following
entries 
$$a_{k \ell}=\left\{\begin{array}{ll}
\frac{1}{\sigma \sqrt{2 \pi}} \exp \left(-\frac{(k-\ell)^{2}}{2 \sigma^{2}}\right), & |k-\ell| \leq r \\
0, & \text { otherwise }
\end{array}\right.$$
Note  that $\sigma$ controls the amount of smoothing, i.e. the larger the $\sigma$,  the more ill posed the problem.  We generated a blurred and noisy tensor image $\mathscr{C}=\widehat{\mathscr{C}}+\mathscr{N},$ where $\mathscr{N}$ is a noise tensor with normally distributed random entries with zero mean and with variance chosen to correspond to a specific noise level $\nu:=\|\mathscr{N}\|_F /\|\widehat{\mathscr{C}}\|_F.$
To determine the effectiveness of our solution methods, we evaluate 
$$\text{Relative error}=\frac{\left\|\hat{ \mathscr X}-{\mathscr  X}_{\text{restored}}\right\|_{F}}{\|\widehat{ \mathscr X}\|_{F}}$$
and the Signal-to-Noise
Ratio (SNR) defined by
\[\text{SNR}({ \mathscr X}_{\text{restored}})=10\text{log}_{10}\frac{\|\widehat{\mathscr X}-E(\widehat{\mathscr X})\|_F^2}{\|{\mathscr X}_{\text{restored}}-\widehat{\mathscr X}\|_F^2},\]
where $E(\widehat{\mathscr X})$ denotes the mean gray-level of the uncontaminated image $\widehat{\mathscr{X}}$. 
All computations were carried out using the MATLAB environment on an Intel(R) Core(TM) i7-8550U CPU @ 1.80GHz (8 CPUs) computer with 12 GB of
RAM. The computations were done with approximately 15 decimal digits of relative
accuracy. 
\subsection{Example 1}
In this example we present the experimental results recovered by Algorithm \ref{TG-GMRES(m)} and Algorithm \ref{TG-GKB} for the reconstruction of a cross-channel blurred color images that have been contaminated by both within and cross blur, and additive noise. The cross-channel blurring is determined by the matrix 
$$\mathbf{A}_{\mathrm{color}}=\left[\begin{array}{ccc}
0.8 & 0.10& 0.10 \\
0.10 & 0.80 & 0.10 \\
0.10& 0.10 & 0.80
\end{array}\right].$$
We consider two $\mathrm{RGB}$ images from \textbf{MATLAB}, $\tt papav256$ ($\widehat{\mathscr X}\in\mathbb{R}^{256\times256\times3}$) and  $\tt peppers$ ($\widehat{\mathscr X}\in\mathbb{R}^{512\times512\times3}$). They are shown on  Figure \ref{fig1}. For the within-channel blurring,  we let $\sigma=4$ and $r=6$. The considered  noise levels are $\nu=10^{-3}$ and $\nu=10^{-2}$. The associated blurred and noisy RGB images $\mathscr{C}=\mathscr{A}\ast\widehat{\mathscr{X}}\ast\mathscr{B}+\mathscr{N}$ for noise level $\nu=10^{-3}$ are shown on  Figure \ref{fig2}.  Given the contaminated RGB image $\mathscr{C}$, we would like to recover an approximation of the original RGB image $\widehat{\mathscr X}$.   The restorations for  noise level $\nu=10^{-3}$ are shown on  Figure \ref{fig3} and they are obtained by applying Algorithm \ref{TG-GMRES(m)} implementing the  Tensor T-Global GMRES method, with  $\mathscr{X}_0=\mathscr{O}$, $tol=10^{-6}$, $m=10$ and $\text{Iter}_\text{max}=10$. Using GCV, the computed optimal value for the projected problem  was $\mu_{10}=3.82\times 10^{-5}.$ Table \ref{tab1} compares, the computing time (in seconds),  the relative errors and the SNR of the computed restorations. Note that in this table, the allowed maximum number of outer iterations for Algorithm \ref{TG-GMRES(m)}  with noise level $\nu=10^{-2}$ was $\text{Iter}_\text{max}=4$ and the maximum number of inner iterations was $m=4$.  The restorations obtained with Algorithm \ref{TG-GKB}  are shown on  Figure \ref{fig4}. For the $\tt papav256$ color image, the discrepancy principle with $\eta=1.1$ is satisfied when $m=64$ steps of the Tensor T-GGKB  method (Algorithm \ref{TG-GK}) have been carried out, producing a regularization parameter given by $\mu_m=5.57\times 10^{-5}$. For comparison with existing approaches in the literature, we report in Table \ref{tab1} the results obtained with the method 
proposed in \cite{CR}. This method utilizes the connection between (standard) 
Golub--Kahan bidiagonalization and Gauss quadrature rules for solving large 
ill-conditioned linear systems of equations (\ref{linmodel}). We refer to this method as
GKB. It determines the regularization parameter analogously to Algorithm \ref{TG-GKB}, and uses a similar stopping criterion. We can see that the
methods yield restorations of the same quality, but the new proposed methods perform significantly better in terms of cpu-time.
\begin{table}[htbp]
	\begin{center}\footnotesize
		\renewcommand{\arraystretch}{1.3}
		\begin{tabular}{cccccc}\hline
			\multicolumn{1}{c}{{RGB images}} &\multicolumn{1}{c}{{Noise level}} & \multicolumn{1}{c}{{Method}}& 
			\multicolumn{1}{c}{{SNR}} & Relative error& CPU-time (sec) \\ 
			\hline 
			\multirow{6}{*}{$\tt papav256$}&\multirow{4}{3em}{$10^{-3}$}&Algorithm \ref{TG-GMRES(m)} &$21.01$&$6.64\times10^{-2}$&$\phantom{1}6.62$\\
			&&Algorithm \ref{TG-GKB}&$20.41$& $7.12\times10^{-2}$&$\phantom{1}5.87$\\
			&&GKB&$20.99$&$7.12\times10^{-2}$&$18.61$\\
			\cline{2-6}
			&\multirow{4}{3em}{$10^{-2}$}&Algorithm \ref{TG-GMRES(m)}&$18.00$&$9.40\times10^{-2}$&$\phantom{1}1.18$\\
			&	&Algorithm \ref{TG-GKB}&$17.78$&$9.64\times10^{-2}$&$\phantom{1}1.11$\\
			&&GKB&$17.78$&$9.64\times10^{-2}$&$5.79$\\
			\hline 
			\multirow{6}{*}{$\tt peppers$}&\multirow{4}{3em}{$10^{-3}$}&Algorithm \ref{TG-GMRES(m)} &$19.39$&$5.50\times10^{-2}$&$\phantom{1}24.32$\\
			&&Algorithm \ref{TG-GKB}&$19.11$& $5.68\times10^{-2}$&$\phantom{1}25.63$\\
			&&GKB&$19.11$&$5.68\times10^{-2}$&$78.13$\\
			\cline{2-6}
			&\multirow{4}{3em}{$10^{-2}$}&Algorithm \ref{TG-GMRES(m)}&$16.23$&$7.92\times10^{-2}$&$\phantom{1}4.59$\\
			&	&Algorithm \ref{TG-GKB}&$15.61$&$8.50\times10^{-2}$&$\phantom{1}3.39$\\
			&&GKB&$15.61$&$8.50\times10^{-2}$&$15.16$\\
			\hline 
		\end{tabular}
		\caption{Results for Example 1.}\label{tab1}
	\end{center}
\end{table}
\begin{figure}
	\begin{center}
		\includegraphics[width=5in]{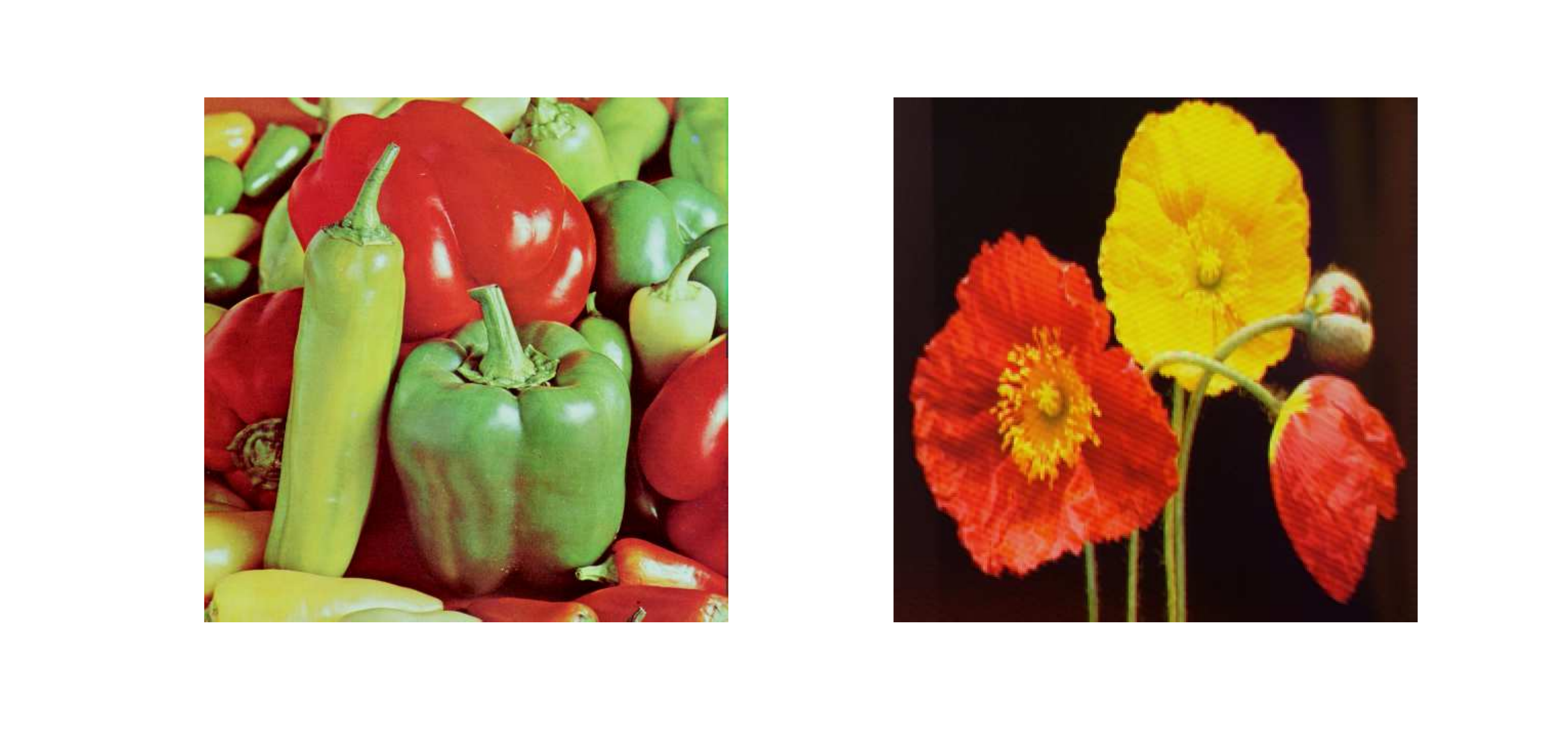}
		\caption{Example 1: Original RGB images:  $\tt peppers$ (left), $\tt papav256$ (right).}\label{fig1}
	\end{center}
\end{figure}

\begin{figure}
	\begin{center}
		\includegraphics[width=5in]{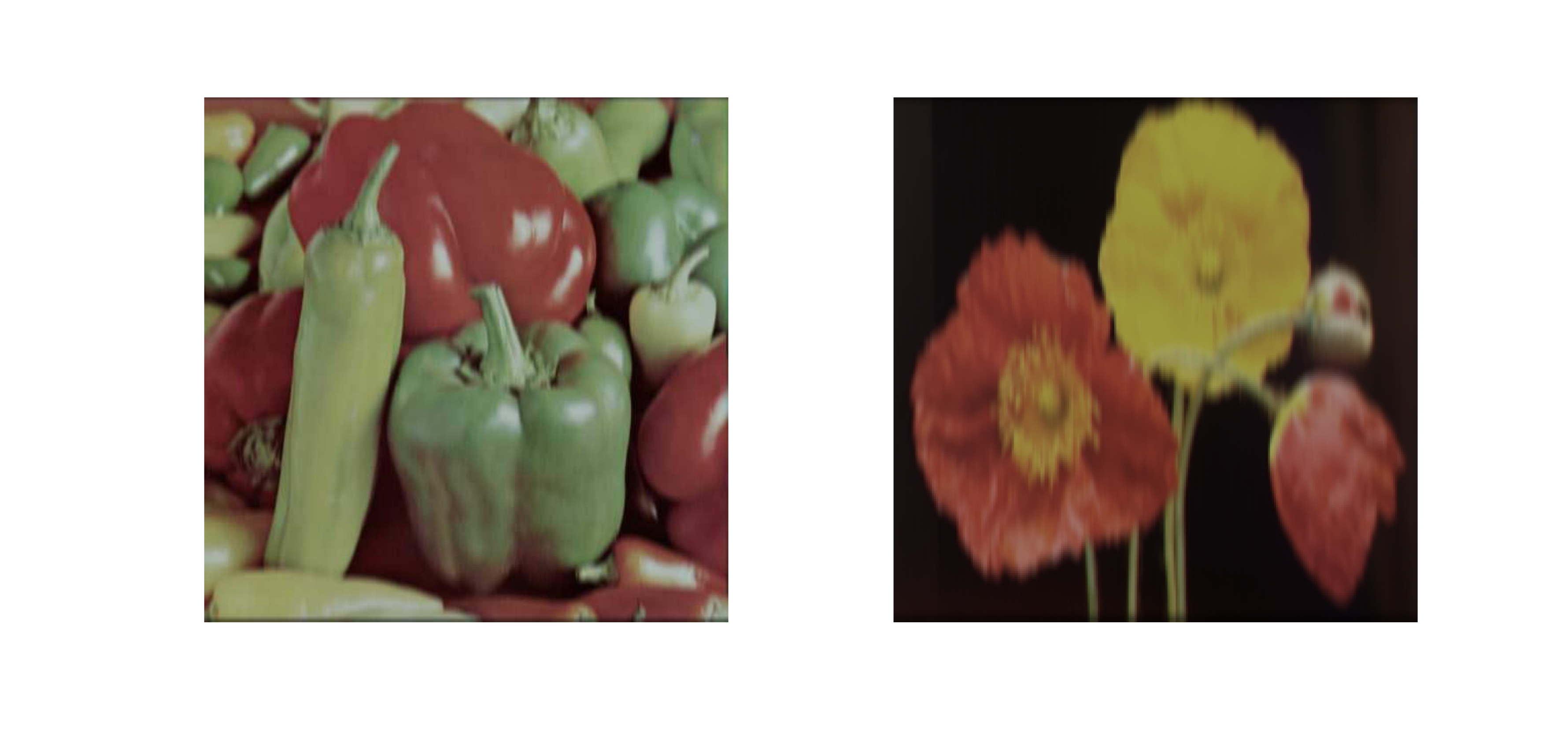}
		\caption{Example 1: Blurred and noisy images, $\tt peppers$ (left), $\tt papav256$ (right).}\label{fig2}
	\end{center}
\end{figure}
\begin{figure}
	\begin{center}
		\includegraphics[width=5in]{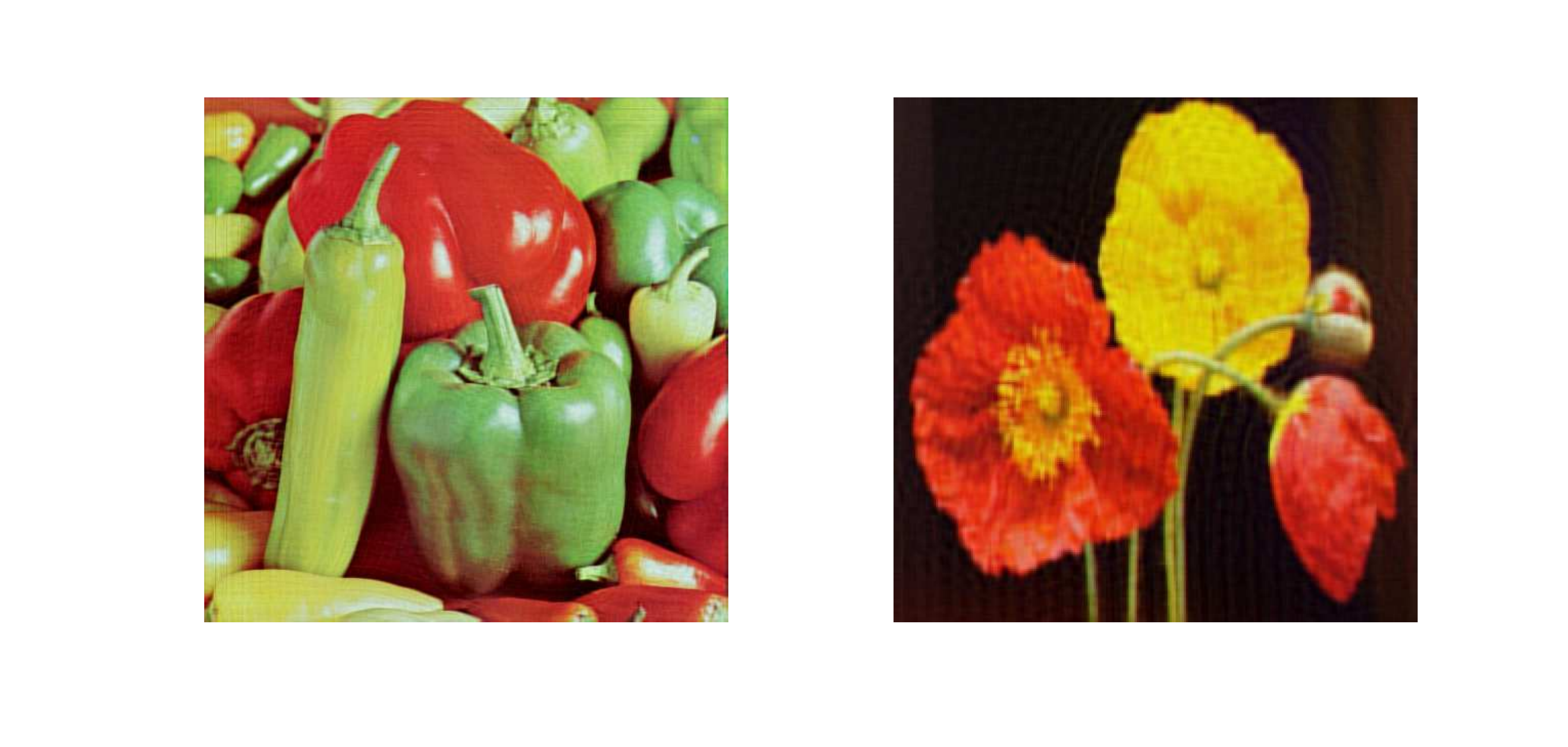}
		\caption{Example 1: Restored images by Algorithm \ref{TG-GMRES(m)}, $\tt peppers$ (left), $\tt papav256$ (right).}\label{fig3}
	\end{center}
\end{figure}
\begin{figure}
	\begin{center}
		\includegraphics[width=5in]{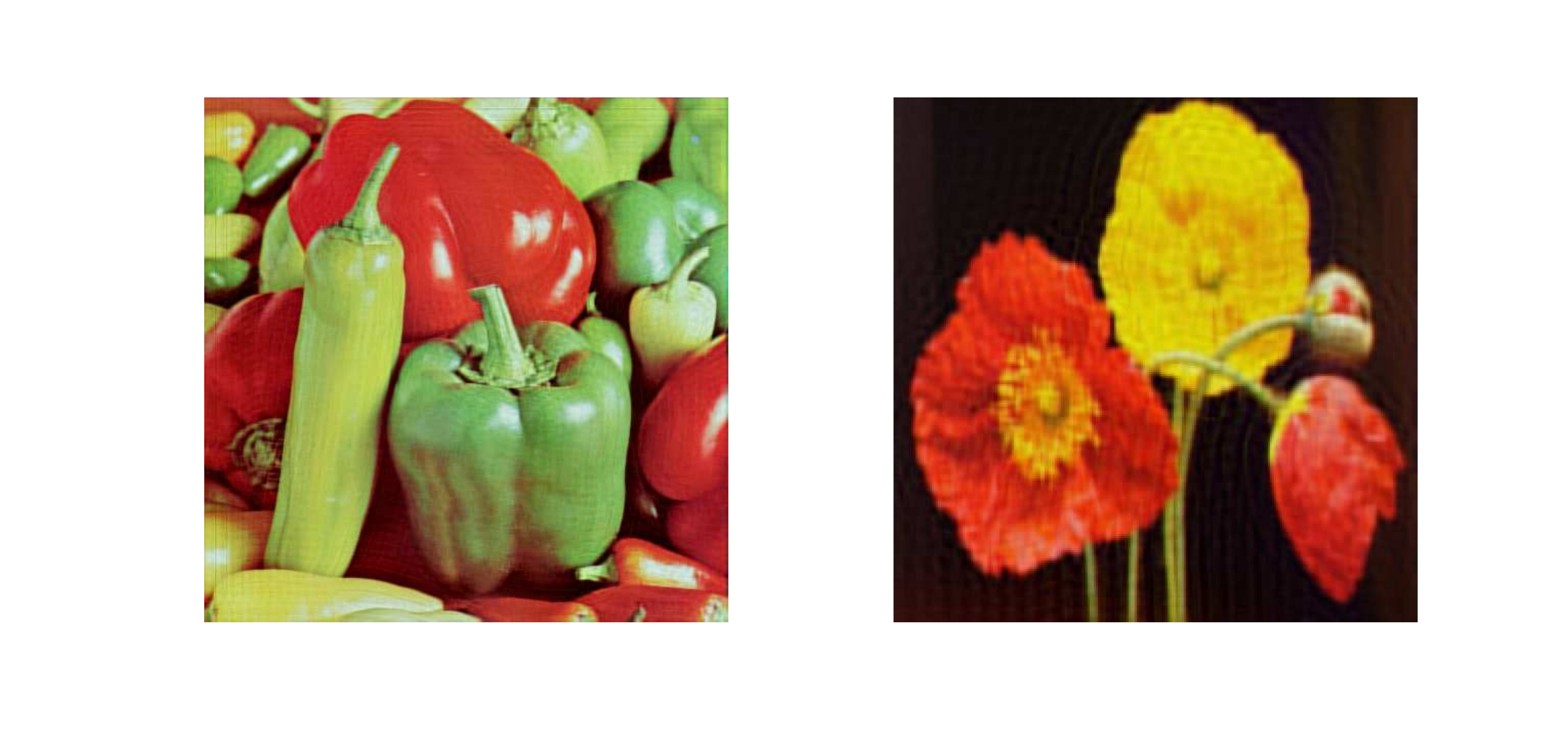}
		\caption{Example 1: Restored images by Algorithm \ref{TG-GKB}, $\tt peppers$ (left), $\tt papav256$ (right).}\label{fig4}
	\end{center}
\end{figure}
\subsection{Example 2} 
In this example, we evaluate the effectiveness of Algorithm \ref{TG-GMRES(m)} and Algorithm \ref{TG-GKB} when applied to the restoration of a color video defined by a sequence of RGB images. Video restoration is the problem of restoring a sequence of $k$ color images (frames). Each frame is represented by a tensor of $n \times n\times3$ pixels. In the present example, we are interested in restoring 10 consecutive frames of a contaminated video. Note that the processing of such given frames, one  at a time, is extremely time consuming. We consider the xylophone video from MATLAB. The video clip is in MP4 format with each frame having $240 \times 240$ pixels. The (unknown) blur- and noise-free frames are stored in the tensor $\widehat{ \mathscr X} \in \mathbb{R}^{240 \times 240\times30}$, obtained by stacking the grayscale images that constitute the three channels of
each blurred color frame. These frames are blurred by $\mathscr {A}\star \widehat{\mathscr{X}}\star \mathscr{B}= \widehat{\mathscr{C}}$, where  $\mathscr {A}$ and $\mathscr {B}$  are a 3-way tensors such that  $\mathscr {A}(:,:,1)= \mathbf{A^{(2)}}$, $\mathscr {B}(:,:,1)= (\mathbf{A^{(1)}})^T$ and $\mathscr {A}(:,:,i)=\mathscr {B}(:,:,i)=0$, for $i=2,...,30$,  using $\sigma=2$ and $r=4$ to build the blurring matrices. We consider white Gaussian noise of levels $\nu=10^{-3}$ or $\nu=10^{-2}$. Figure \ref{frame5ob} shows the 5th exact (original) frame and the contaminated version with noise level $\nu=10^{-3}$, which is to be restored. Table \ref{tab2} displays the performance of Algorithm \ref{TG-GMRES(m)} and Algorithm \ref{TG-GKB}. In Algorithm \ref{TG-GMRES(m)}, we have used as an input for noise level $\nu=10^{-3}$,  $\mathscr{C}$, $\mathscr{X}_0=\mathscr{O}$, $tol=10^{-6}$, $m=10$ and $\text{Iter}_{\text{max}}=10$.The chosen inner and outer iterations for noise level $\nu=10^{-2}$ were $m=4$ and  $\text{Iter}_{\text{max}}=4$, respectively. For the ten outer iterations,  minimizing the GCV function  produces  $\mu_{10}=1.15 \times 10^{-5}$. Using Algorithm \ref{TG-GKB},  the discrepancy principle with $\eta=1.1$ have been satisfied after $m=59$ steps of  T-GGKB method (Algorithm \ref{TG-GK}), producing a regularization parameter given by $\mu_m=1.06\times10^{-4}$. For completeness, the restorations obtained with Algorithm \ref{TG-GMRES(m)} and Algorithm \ref{TG-GKB} are shown on the left-hand and the right-hand side of Figure \ref{frame5r}, respectively.
\begin{table}[htbp]
	\caption{Results for Example 2.}\label{tab2}
	\begin{center}
		\begin{tabular}{lcccc}
			\hline Noise level & Method & SNR & Relative error & CPU-time (second) \\
			\hline {$10^{-3}$}& Algorithm \ref{TG-GMRES(m)} & 19.97 & $4.07 \times 10^{-2}$ & 35.68 \\  &  Algorithm \ref{TG-GKB} & 19.24 & $4.43 \times 10^{-2}$ & 25.52 \\
			\hline  {$10^{-2}$} & Algorithm \ref{TG-GMRES(m)}  & 15.17 & $7.08 \times 10^{-2}$ & 6.12 \\
			&Algorithm \ref{TG-GKB} & 15.13 & $7.11 \times 10^{-2}$ & 4.40 \\
			\hline
		\end{tabular}
	\end{center}
\end{table}
\begin{figure}
	\begin{center}
		\includegraphics[width=5in]{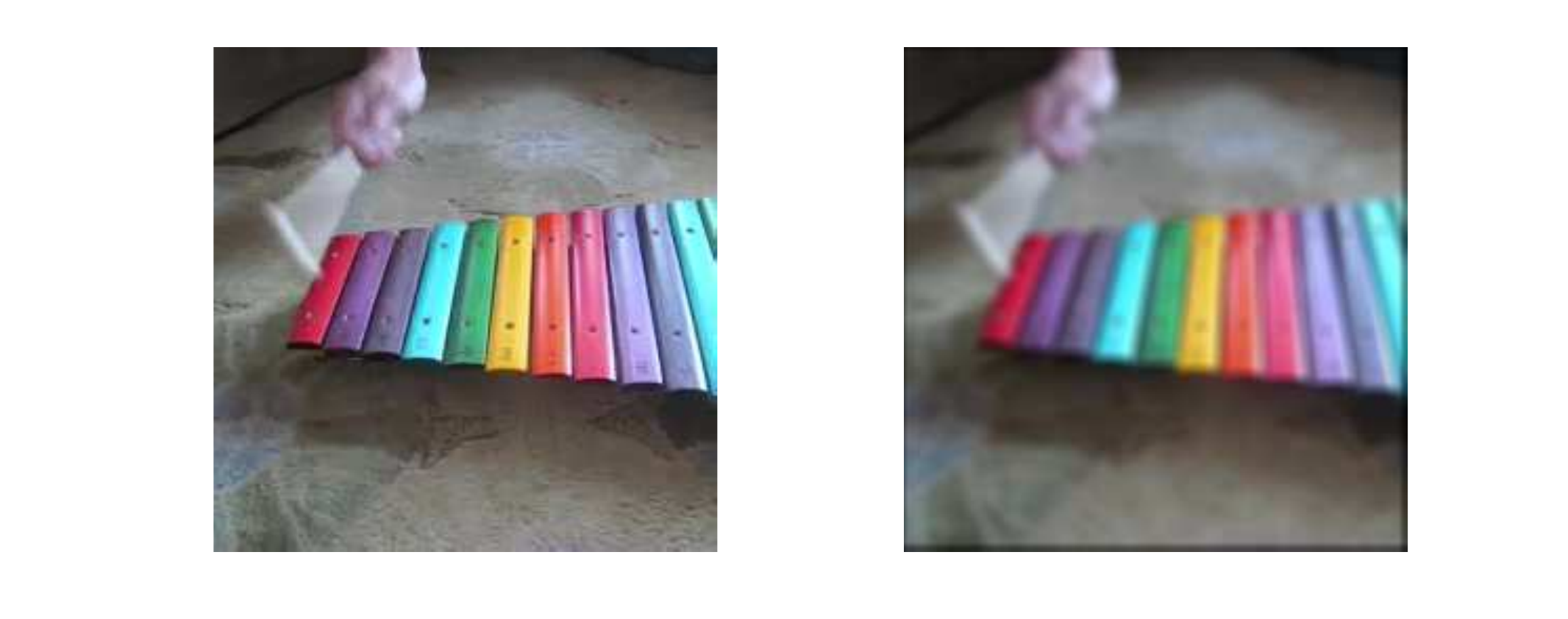}
		\caption{Example 2: Original frame no. 5  (left), blurred and noisy frame no. 5  (right).
		}\label{frame5ob}
	\end{center}
\end{figure}
\begin{figure}
	\begin{center}
		\includegraphics[width=5in]{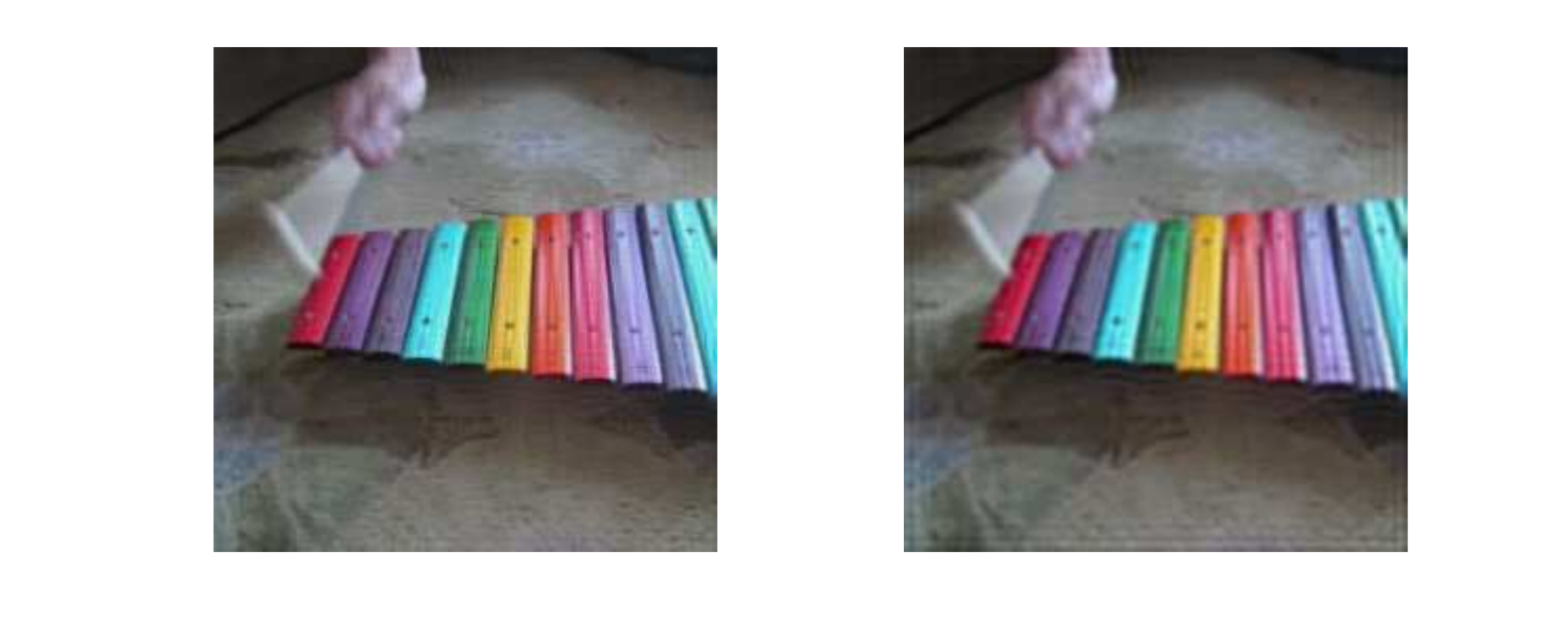}
		\caption{Example 2: restored frame no. 5 by Algorithm \ref{TG-GMRES(m)}  (left), and restored frame no. 5 by Algorithm \ref{TG-GKB} 
			(right).}\label{frame5r}
	\end{center}
\end{figure}
\section{Conclusion} 
In this paper we have proposed tensor version of GMRES and Golub–Kahan bidiagonalization
algorithms using the T-product, with applications to  solving  large-scale linear tensor equations arising in the reconstructions of blurred and noisy multichannel images and videos. The numerical experiments that we have performed  show the
effectiveness of the proposed schemes to inexpensively computing regularized solutions of high quality.
\bibliographystyle{siam}
\bibliography{TMatrix_2020}  
\end{document}